\pdfoutput=1
\documentclass[a4paper,12pt,reqno]{amsart}
\usepackage{amscd,amssymb,amsmath,amsfonts,amsthm, mathrsfs,xspace}
\usepackage{graphicx}
\usepackage{tensor}
\usepackage{verbatim,enumerate,paralist}
\usepackage{textcomp}
\usepackage[pagebackref,colorlinks]{hyperref}
\hypersetup{%
  urlcolor=blue,linkcolor=blue,citecolor=blue
}
\usepackage{pdfpages}

\DeclareMathAlphabet{\mathbbe}{U}{bbold}{m}{n}

\newcommand{\thg}{{\mathord{\text{--}}}}
\newcommand{\abs}[1]{{\left|{#1}\right|}}

\renewcommand{\epsilon}{\varepsilon}
\renewcommand{\phi}{\varphi}

\usepackage[T1]{fontenc}
\usepackage[all]{xy}

\newcounter{Definitioncount}
\newtheorem{theorem}{Theorem}[section] % with counter
\newtheorem{lemma}[theorem]{Lemma}
\newtheorem{proposition}[theorem]{Proposition}

\theoremstyle{definition}
\newtheorem{examples}[theorem]{Examples}

\newcommand{\spn}[1]{{\langle{#1}\rangle}}

\newdir{ >}{\xybox{*{}+/2.5mm/*@{>}}} % can go in the preamble

\newtheoremstyle{fact}{\bigskipamount}{\medskipamount}{\upshape}{}{\itshape}{. }{ }{Fact}
\theoremstyle{fact}
\newtheoremstyle{genquest}{\bigskipamount}{\medskipamount}{\upshape}{}{\itshape}{. }{ }{General Question}
\theoremstyle{genquest}
\newtheoremstyle{step}{2\bigskipamount}{\medskipamount}{\upshape}{}{\itshape}{. }{ }{\underline{Step~\thestep}}
\theoremstyle{step}

\renewcommand{\thestep}{\arabic{step}}

\newcommand{\dd}{\colon}
\newcommand{\lra}{\longrightarrow}

\newcommand{\ldual}[1]{\mathord{{\let\nolimits\relax\sideset{^\wedge}{}{#1}}}}
\newcommand{\laction}[2]{\mathord{{\let\nolimits\relax\sideset{^{#1}}{}{#2}}}}
\newcommand{\conj}[2]{\mathord{{\let\nolimits\relax\sideset{^{#1}}{}{#2}}}}

\def\CA{{\mathscr A}}
\def\CB{{\mathscr B}}
\def\CC{{\mathscr C}}
\def\CD{{\mathscr D}}
\def\CE{{\mathscr E}}

\def\CM{{\mathscr M}}

\def\CP{{\mathscr P}}

\def\CR{{\mathscr R}}

\def\CV{{\mathscr V}}

\def\CX{{\mathscr X}}

\renewcommand{\phi}{\varphi}
\renewcommand{\epsilon}{\varepsilon}

\def\1c#1{\stackrel{#1}{\to}}

\DeclareMathAlphabet{\mathbf}{OT1}{cmr}{b}{n}

\begin{document}

\author{Richard Garner}
\address{Department of Mathematics, Macquarie University, NSW 2109, Australia}
\email{richard.garner@mq.edu.au}

\author{Ross Street}
\address{Department of Mathematics, Macquarie University, NSW 2109, Australia}
\email{ross.street@mq.edu.au}

\thanks{Both authors gratefully acknowledge the support of the
  Australian Research Council Discovery Grant DP130101969; Garner
  acknowledges with equal gratitude the support of an Australian
  Research Council Discovery Grant DP110102360.}

\title[Coalgebras for weighted and pointwise products]{Coalgebras governing both weighted Hurwitz products and their pointwise transforms}
\date{\today}

\maketitle

\begin{abstract}
  We give further insights into the weighted Hurwitz product and the
  weighted tensor product of Joyal species. Our first group of results
  relate the Hurwitz product to the pointwise product, including the
  interaction with Rota--Baxter operators. Our second group of results
  explain the first in terms of convolution with suitable bialgebras,
  and show that these bialgebras are in fact obtained in a
  particularly straightforward way by freely generating from pointed
  coalgebras. Our third group of results extend this from linear
  algebra to two-dimensional linear algebra, deriving the existence of
  weighted Hurwitz monoidal structures on the category of species
  using convolution with freely generated bimonoidales. Our final
  group of results relate Hurwitz monoidal structures with
  equivalences of of Dold--Kan type.
\end{abstract}

%%\tableofcontents

\section{Introduction}
\label{sec:introduction}

This paper continues the investigations of~\cite{wtp} into the
$\lambda$-Hurwitz products of~\cite{GKZ2008,GKZ2014}. Given a ring
$\mathrm{k}$, an element $\lambda \in \mathrm{k}$, and a
$\mathrm{k}$-algebra $A$, the \emph{$\lambda$-Hurwitz product} is a
certain multiplication $\cdot^\lambda$ on the set $A^\mathbb N$ which,
together with the pointwise linear structure, endows it with the
structure of a $\mathrm{k}$-algebra $G_\lambda A$. This algebra has a
universal role: it is the cofree $\lambda$-differential algebra on
$A$. Here, a \emph{$\lambda$-differential algebra} is a
$\mathrm{k}$-algebra equipped with a \emph{$\lambda$-weighted
  derivation}---a $\mathrm{k}$-linear endomorphism $\partial$
satisfying
\begin{equation}\label{deriv}
  \partial(1) = 0 \qquad \text{and} \qquad \partial(ab) = (\partial a)b + a(\partial b) + \lambda (\partial a) (\partial b)\rlap{ .}
\end{equation}
When $\lambda = 0$, of course, we re-find the classical notion of
derivation and differential algebra; when $\lambda \neq 0$, we have
variants on these notions apt for the study of \emph{difference}
rather than differential equations.

Our first set of results clarify the relation between the pointwise
and $\lambda$-Hurwitz products on $A^\mathbb N$. We exhibit algebra
morphisms
$\gamma \colon (A^\mathbb N, \cdot^\lambda) \rightarrow (A^\mathbb N,
\text{pointwise})$,
which are algebra \emph{isomorphisms} whenever
$\lambda \in \mathrm{k}$ is invertible. We relate these to Lagrange
interpolation, and also to weighted \emph{Rota--Baxter
  algebras}~\cite{Baxt1960,Rota1969}: these are $\mathrm{k}$-algebras
endowed with a \emph{weighted Rota--Baxter operator}---a
$\mathrm{k}$-linear endomorphism $P$ satisfying the equation:
\begin{equation}\label{RBo}
  P(a)P(b) = P(P(a)b +a P(b)+\lambda ab)\rlap{ .}
\end{equation}
Just as derivations encode abstract differentation, so weighted
Rota--Baxter operators encode abstract integration (when
$\lambda = 0$) or summation (when $\lambda \neq 0$). Following
\cite{GKZ2008,GKZ2014}, we show that, when $A$ is a weighted
Rota--Baxter algebra, $\gamma$ above lifts to a homomorphism
of Rota--Baxter algebras; in particular, when $\lambda$ is invertible,
this establishes an isomorphism between two canonical Rota--Baxter
algebra structures on $A^\mathbb N$ arising from a given one on $A$.

Our second set of results explain the first in terms of more basic
data. The assignation $A \mapsto G_\lambda A$ underlies a comonad on
the category of $\mathrm{k}$-algebras; in fact, this may be seen as
induced by convolution with a $\mathrm{k}$-bialgebra
$C(\lambda)_\infty$. There is another comonad $H$ on this category
given by $A \mapsto (A^\mathbb N, \text{pointwise})$, which is again
induced by convolution with a bialgebra $D_\infty$. Now the
$\mathrm{k}$-algebra morphism
$\gamma \colon G_\lambda A \rightarrow HA$ can be seen as induced
under convolution by a morphism of bialgebras
$D^\infty \rightarrow C(\lambda)_\infty$. In fact, more is true. The
comonads $G_\lambda$ and $H$ are both \emph{cofree} on a copointed
endofunctor; correspondingly, the bialgebras $D_\infty$ and
$C(\lambda)_\infty$ are free on pointed coalgebras $D$ and
$C(\lambda)$, and in these terms, the bialgebra morphism
$D_\infty \rightarrow C(\lambda)_\infty$ can be seen as generated by
the (much simpler) datum of a morphism of pointed coalgebras
$D \rightarrow C(\lambda)$.

The remaining contributions of this paper are concerned with
``categorifications'' of the preceding ones. Rather than considering
modules over a commutative ring $\mathrm{k}$, we consider categories
enriched over a suitable symmetric monoidal base $\CV$ admitting a
suitable class of colimits that play the role of ``addition''. Rather
than (commutative) $\mathrm{k}$-algebra structure, we consider
(symmetric) monoidal structure on our $\CV$-categories; and rather
than \emph{coalgebra} structure, we consider \emph{comonoidale}
structure in a suitable monoidal bicategory of $\CV$-categories. As
described in~\cite{wtp}, there is in this setting a
``categorification'' of the $\lambda$-Hurwitz product found on
$A^\mathbb N$ for any $\mathrm{k}$-algebra $A$ to a
\emph{$\Lambda$-Hurwitz monoidal structure} on $\CA^\mathfrak S$
(where $\mathfrak S$ is the category of finite sets and bijections)
for any ``$\CV$-algebra'' $\CA$. Inspired by the constructions of the
preceding sections, our third set of results exhibit this
$\Lambda$-Hurwitz monoidal structure as induced by convolution with a
``$\CV$-bialgebra'', and show that this bialgebra may in fact be
obtained as the free bialgebra generated by a pointed $\CV$-coalgebra.
Analogously to before, we obtain a comparison between the Hurwitz
monoidal structure on $\CA^\mathfrak S$ and the pointwise one, which,
once again, may be seen as freely generated from a morphism of pointed
$\CV$-coalgebras.

The final set of results in this paper explain the link between
categorified Hurwitz tensor products and \emph{equivalences of
  Dold--Kan type}. The classical Dold--Kan equivalence is that between
simplicial abelian groups and chain complexes of abelian groups, the
simpler direction of which is the functor
$N \colon [\Delta^\mathrm{op}, \mathbf{Ab}] \rightarrow \mathbf{Ch}$
sending each simplicial abelian group to its normalized Moore complex.
However, when we equip $[\Delta^\mathrm{op}, \mathbf{Ab}]$ with the
pointwise tensor product, and $\mathbf{Ch}$ with its classical tensor
product, the functor $N$ is \emph{not} strong monoidal, though it is
lax and oplax monoidal in a compatible manner:
see~\cite[Chapter~5]{AguiarMahajan2010}. This means that transporting
the pointwise monoidal structure on
$[\Delta^\mathrm{op}, \mathbf{Ab}]$ across this equivalence yields a
new tensor product on chain complexes, and as we will see, the formula
for this is precisely the $\lambda = 1$ case of a Hurwitz-style
tensor\footnote{The explicit calculation of this tensor product
  appears in unpublished work of Lack and Hess.}.

In fact, we will show something more general than this. Recent work
such as~\cite{Pirashvili2000Dold-Kan,Slominska,Church2015FI-modules}
has established various generalisations of the classical Dold--Kan
equivalence; in~\cite{Lack2015Combinatorial} is described a general
framework for obtaining such equivalences, which, starting from a
category $\CP$ equipped with suitable extra structure, derives a
category with zero morphisms $\CD$ and an equivalence of functor
categories $[\CP, \mathbf{Ab}]\simeq [\CD, \mathbf{Ab}]_\mathrm{pt}$
(here the subscript ``pt'' indicates the restriction to zero-map
preserving functors). Our fourth main result shows that, in this
setting, the pointwise tensor product on $[\CP,\mathbf{Ab}]$ always
transports to a Hurwitz-style monoidal structure on
$[\CD, \mathbf{Ab}]_\mathrm{pt}$; while our fifth and final result
shows that certain important examples of equivalences arising in this
way, may, as before, be seen as induced by convolution with a map of
``$\mathbf{Ab}$-bialgebras'' freely generated from a map of pointed
$\mathbf{Ab}$-coalgebras.

\section{Preliminaries}

Throughout this paper, $\mathrm{k}$ will be a commutative
$\mathbb{Q}$-algebra. Given natural numbers $n, m_1, \dots, m_r$, we
define the \emph{multinomial coefficient}
\begin{align*}
\smash{\binom{n}{m_1,\dots,m_r}} = \frac{n!}{m_1!\cdots m_r!} \ .
\end{align*}
Usually this is for $\Sigma_i m_i =n$, so that this coefficient gives the
number of ways of partitioning a set of cardinality $n$ into disjoint
subsets of cardinalities $m_1, \dots, m_r$. We extend this definition
to all integers by declaring $k!$ to be zero for any
$k < 0 \in \mathbb Z$. Of course, we write $\binom{n}{r}$ as usual for
$\binom{n}{r,n-r}$; more generally, for a $\mathrm{k}$-algebra $A$ and
$x \in A$, we define the \emph{binomial coefficients} of $x$ to be the
elements of $A$ given by:
\begin{align*}
  \binom{x}{0} = 1 \qquad \text{and} \qquad 
  \binom{x}{r} = \frac{x(x-1)\cdots (x-r+1)}{r!} \quad \text{ for $0 < r \in \mathbb N$ .}
\end{align*}

The following elementary result is classical.

\begin{lemma}\label{combident}
  Let $p, q \in \mathbb N$.
  \begin{enumerate}[(i)]
  \item If $n \in \mathbb N$, then
    \begin{equation*}
      \binom{n}{p}\binom{n}{q} =\sum_{\substack{u+r+s+t=n\\ p=r+t, \, q=s+t}}{\binom{n}{u,r,s,t}}\rlap{ ;}
    \end{equation*}
  \item If $A$ is a $\mathrm{k}$-algebra and $x \in A$, then
    \begin{equation*}
      \binom{x}{p}\binom{x}{q} =\sum_{t}{\binom{p+q-t}{p-t,q-t,t}\binom{x}{p+q-t}}\rlap{ .}
    \end{equation*}
    % \item We have
    %   \begin{equation*}
    %     \binom{p+1}{q+1} = \sum_{m = 0}^p \binom{m}{q}\rlap{ ;} 
    %   \end{equation*}
  \end{enumerate}
\end{lemma}

\begin{proof}
  For a finite set $X$, we have bijections between any two of the sets
  \begin{align*}
    \{\,P,Q \subseteq X & : \abs P=p,\abs Q=q\,\} \ ,\\
    \{\,U,R,S,T \subseteq X & : X =U+R+S+T, \abs {R+T}=p,\abs {S+T}=q\,\} \ ,\\
    \text{and}\quad \{\,W,R,S,T \subseteq X & : W = R+S+T, \abs{R+T}=p,\abs{S+T}=q\,\} \ .
  \end{align*}
  This proves (i) and also (ii) in the case where $x \in \mathbb N$. The
  general case of (ii) proceeds by a straightforward induction on $p$.
\end{proof}

We will also require the following simple combinatorial identity.

\begin{lemma}\label{combseq}
  For any $\mathrm{k}$-algebra $A$ and sequence $a_0, a_1, \dots$ in
  $A$, it holds that:
  \begin{equation*}
    a_n = \sum_{r+s+t=n}{\binom{n}{r,s,t} (-1)^sa_t} \ .
  \end{equation*}
\end{lemma}
\begin{proof}
  Let $\Sigma \dd A^{\mathbb{N}}\to A^{\mathbb{N}}$ be the suspension
  operator $\Sigma(a)_n=a_{n+1}$. As the operators $1, -1, \Sigma$
  commute with each other, we have by the trinomial formula that
  \begin{equation*}
    \Sigma^{\circ n} = (1-1+\Sigma)^{\circ n}= \sum_{r+s+t=n}{\binom{n}{r,s,t} 1^{\circ r} \circ (-1)^{\circ s} \circ \Sigma^{\circ t}} \ .
  \end{equation*}
  Applying this operator identity to $a\in A^{\mathbb{N}}$ and
  evaluating at $0$ yields the result.
\end{proof}

\section{The $\lambda$-weighted Hurwitz product}\label{wHp}

Throughout this section, we fix $\lambda \in \mathrm{k}$ and a
$\mathrm{k}$-algebra $A$. The {\em $\lambda$-weighted Hurwitz product}
on $A^{\mathbb{N}}$ \cite[\S 2.3]{GKZ2008} is defined by the equation
\begin{equation}\label{LambHurw}
  (a\cdot^{\lambda}b)_n=\sum_{n=r+s+t}{\binom{n}{r,s,t}\lambda^t a_{r+t} b_{s+t}} \ .
\end{equation}
This product has $(1,0,0,\dots)$ as neutral element; taken together
with the pointwise $\mathrm{k}$-linear structure we obtain a
$\mathrm{k}$-algebra $(A^{\mathbb{N}},\cdot^{\lambda})$, which is
commutative whenever $A$ is so. This formula restricts to $A^{\ell}$
regarded as the linear subspace of $A^{\mathbb{N}}$ comprising those
$a$ with $a_n=0$ for $n\geqslant \ell$; moreover, we can recapture
$A^\mathbb N$ and its algebra structure from the $A^\ell$'s as the
limit of the chain
$\dots \rightarrow A^n \rightarrow \dots \rightarrow A^1 \rightarrow
A^0$,
where each map $A^{j+1} \rightarrow A^{j}$ sets $a_j$ to zero. In
light of this, we may consider that $A^{\mathbb{N}}=A^{\ell}$ for
$\ell = \infty$.

Our first result relates the $\lambda$-weighted algebra structure on
each $A^\ell$ to the pointwise one. Its first part allows us to focus
attention on the case $\lambda = 1$ as occurring in \cite{Rota1969};
the second reduces that to the pointwise case.

\begin{proposition}
  \label{prop:1}
  Let $\ell \in \mathbb N \cup \{\infty\}$.
  \begin{enumerate}[(i)]
  \item There is a $\mathrm{k}$-algebra morphism
    $\hat{\lambda} \colon (A^\ell, \cdot^\lambda) \rightarrow (A^\ell,
    \cdot^1)$ defined by $\hat{\lambda}(a)_n=\lambda^na_n$;
  \item There is a $\mathrm{k}$-algebra isomorphism
    $\theta_{\ell}\dd (A^{\ell},\cdot^{1})\rightarrow
    (A^{\ell},\mathrm{pointwise})$ defined by
    \begin{equation*}
      \theta_{\ell}(a)_n = \sum_{m}{\binom{n}{m}a_m }\rlap{ ;}
    \end{equation*}
  \item There is a $\mathrm{k}$-algebra morphism
    $\gamma_\ell \colon (A^\ell, \cdot^\lambda) \rightarrow (A^\ell,
    \text{pointwise})$ defined by
    \begin{equation*}
      \gamma_{\ell}(a)_n = \sum_{m}{\binom{n}{m}\lambda^m a_m }\rlap{ ,}
    \end{equation*}
    which is an isomorphism whenever $\lambda$ is invertible in
    $\mathrm{k}$.
  \end{enumerate}
\end{proposition}

\begin{proof}
  (iii) is immediate from (i) and (ii). For (i), clearly $\hat \lambda$ is linear and preserves the multiplicative
  unit; we conclude since
  \begin{align*}
    \hat \lambda(a \cdot^\lambda b)_n 
    & = \sum_{n=r+s+t}{\binom{n}{r,s,t}\lambda^{n+t} a_{r+t} b_{s+t}}\\
    & = \sum_{n=r+s+t}{\binom{n}{r,s,t}\lambda^{r+t} a_{r+t} \lambda^{s+t} b_{s+t}} 
    = ({\hat \lambda} a \cdot^1 {\hat \lambda} b)_n\rlap{ .}
  \end{align*}
  For (ii), clearly $\theta_\ell$ is linear and preserves $1$; for 
  the multiplication, we calculate that:
  \begin{align*}
    \theta_{\ell}(a\cdot^{1}b)_n 
    & = \sum_{m}\,{\sum_{r+s+t=m}{\binom{n}{m}\binom{m}{r,s,t} a_{r+t} b_{s+t}}} \\
    & = \sum_{u+r+s+t=n}{\binom{n}{u,r,s,t} a_{r+t} b_{s+t}} \\
    &= \sum_{p,q=0}^n{\binom{n}{p}\binom{n}{q} \ a_{p} b_{q}} 
    = (\theta_{\ell}(a) \theta_{\ell}(b))_n
  \end{align*}
  using Lemma~\ref{combident} at the third step. Now $\theta_{\ell}$
  is invertible since it is represented on the standard basis of
  $A^{\ell}$ by a triangular matrix with $1$'s along the main
  diagonal; a direct calculation hinging on Lemma~\ref{combseq} shows that
  an explicit inverse $\bar{\theta}_{\ell}$ is given by
  \begin{equation*}
    \bar{\theta}_{\ell}(a)_m = \sum_{n}{\binom{m}{n}(-1)^{m-n}a_n} \ . \qedhere
  \end{equation*}
\end{proof}

We may also understand the ring isomorphism of
Proposition~\ref{prop:1}(ii) in terms of Lagrange interpolation. Write
$A[x]$ for the $\mathrm{k}$-algebra of polynomials in indeterminate
$x$ with coefficients in $A$. For each natural number $\ell$, we
consider the quotient ring $A[x]/\binom{x}{\ell}$. In this ring we
have $\binom{x}{n}=0$ for all $n\geq \ell$. When $\ell = \infty$, we
define $A[x] / \binom{x}{\ell}$ to be the limit of the sequence of
quotient maps
$\cdots \rightarrow A[x] / \binom{x}{n} \rightarrow \cdots \rightarrow
A[x] / \binom{x}{0}$.

\begin{proposition}
  Let $\ell \in \mathbb N \cup \{\infty\}$.
  \begin{enumerate}[(i)]
  \item There is a $\mathrm{k}$-algebra morphism
    $\psi_{\ell}\dd (A^{\ell},\cdot^{1})\lra A[x]/\binom{x}{\ell}$ defined
    by
    \begin{equation*}
      \psi_{\ell}(a) = \sum_{n}{a_n \binom{x}{n}} \rlap{ ;}
    \end{equation*}
  \item There is a $\mathrm{k}$-algebra isomorphism
    $\phi_{\ell} \colon A[x]/\binom{x}{\ell} \lra
    (A^{\ell},\mathrm{pointwise})$
    defined by $\phi_{\ell}(f)_n = f(n)$ for
    $0 \leqslant n \leqslant \ell$. Moreover, the following triangle
    commutes, implying $\psi_{\ell}$ invertible.
    \begin{equation*}
      \xymatrix{
        (A^{\ell},\cdot^{1}) \ar[rd]_{\theta_{\ell}}\ar[rr]^{\psi_{\ell}} && A[x]/\binom{x}{\ell} \ar[ld]^{\phi_{\ell}} \\
        & (A^{\ell},\mathrm{pointwise}) &
      }
    \end{equation*}
  \end{enumerate}
\end{proposition}

\begin{proof}
  For (i), $\psi_\ell$ is clearly linear and preserves $1$; we
  conclude since
  \begin{align*}
    \psi_{\ell}(a\cdot^{1}b) 
    & = \sum_{n}{\sum_{r+s+t=n}{\binom{n}{r,s,t} a_{r+t} b_{s+t}} \binom{x}{n}} \\
    &= \sum_{r,s,t}{\binom{r+s+t}{r ,s,t}\binom{x}{r+s+t} a_{r+t} b_{s+t}} \\
    &= \sum_{p,q,t}{\binom{p+q-t}{p-t,q-t,t}\binom{x}{p+q-t} a_{p} b_{q}} \\
    &= \sum_{p,q}{a_{p} b_{q}\binom{x}{p}\binom{x}{q} } 
    = \psi_{\ell}(a) \psi_{\ell}(b)
  \end{align*}
  using Lemma~\ref{combident} at the fourth step. When $\ell < \infty$,
  (ii) follows since, by the Chinese Remainder Theorem for rings (see
  \cite{JacobsonBAII} for example), the homomorphism
  \begin{equation*}
    A[x] \rightarrow \prod_{0\le n< \ell}{A[x]/(x-n)} \cong (A^\ell, \text{pointwise})
  \end{equation*}
  obtained by pairing together the canonical quotient maps is itself a
  quotient, with as kernel the ideal generated by
  $x(x-1) \cdots (x - \ell + 1)$. The case $\ell = \infty$ follows on
  passing to the limit.
\end{proof}

We now relate these results to Rota--Baxter operations. As in the
introduction, a {\em Rota--Baxter operator of weight $\lambda$} on a
$\mathrm{k}$-algebra $A$ is a $k$-linear map $P\dd A \to A$
satisfying~\eqref{RBo}. Note that the zero operator is always a
Rota--Baxter operator.

\begin{proposition}\label{lift} 
  Let $\ell \in \mathbb N \cup \{\infty\}$. Each Rota--Baxter operator
  $P$ of weight $\lambda$ on $A$ lifts to one $\bar P$ on
  $(A^\ell, \cdot^\lambda)$, as defined left below, and to one
  $\tilde P$ on $(A^\ell, \text{pointwise})$, as defined right below.
  \begin{equation*}
    \bar P(a)_n =
    \begin{cases}
      P(a_0) & \text{for $n=0$;} \\
      a_{n-1} & \text{for $n>0$,}
    \end{cases}
    \qquad \qquad 
    \tilde P(a)_n = P(a_0) + \lambda \sum_{i < n} a_i\rlap{ .}
  \end{equation*}
  Moreover, the map $\gamma_\ell$ of Proposition~\ref{prop:1}(iii) is a
  map of Rota--Baxter algebras in the sense that
  $\gamma_\ell \bar P = \tilde P \gamma_\ell$.
\end{proposition}

\begin{proof}
  $\bar P$ is a Rota--Baxter operator
  by~\cite[Proposition~3.8]{GKZ2014}. For $\tilde P$, consider the
  difference operator $\partial \colon A^\ell \rightarrow A^\ell$
  defined by $\partial(a)_n = a_{n+1} - a_n$ (taking $a_{\ell+1} = 0$
  when $\ell < \infty$). We have $\partial \circ \tilde P = \lambda$
  and
  $\partial(ab) = (\partial a)b + a(\partial b) + (\partial
  a)(\partial b)$
  under the pointwise product. Clearly $a = b \in A^\mathbb N$ just
  when $a_0 = b_0$ and $\partial(a) = \partial(b)$; thus, since
  \begin{gather*}
    \tilde P(a)\tilde P(b)_0 = P(a_0)P(b_0) = P(P(a_0)b_0 + a_0P(b_0) + \lambda a_0b_0) = 
    P(P(a)b + aP(b) + \lambda ab)_0 \\
    \text{and} \quad \partial(\tilde P(a)\tilde P(b)) = \lambda a \tilde P(b) + \lambda \tilde P(b) a + \lambda^2 ab = \partial \tilde P(a \tilde P(b) + \tilde P(b) a + \lambda ab)
  \end{gather*}
  we conclude that $\tilde P$ is a Rota--Baxter operator as required.
  To see that $\gamma_\ell$ is a map of Rota--Baxter algebras, we
  calculate similarly that
  $\gamma_\ell \bar P(a)_0 = P(a_0) = \tilde P \gamma_\ell(a)_0$, and
  that
  \begin{align*}
    \partial \gamma_\ell\bar P(a)_n & = \sum_m \binom{n+1}{m} \lambda^m (\bar Pa)_m - \sum_m \binom{n}{m} \lambda^m (\bar Pa)_m \\
    & =\sum_m \binom{n}{m-1} \lambda^m (\bar Pa)_m = \sum_{m} \binom{n}{m} \lambda^{m+1} (\bar Pa)_{m+1} \\
    & =\lambda \sum_m \binom{n}{m} \lambda^m a_m = \lambda \gamma_\ell(a)_n = \partial \tilde P \gamma_\ell(a)_n\rlap{ .}\qedhere
  \end{align*}
\end{proof}

\section{Comonadic aspects}
\label{sec:comonadic-aspects}

As described in the introduction, the algebra
$G_\lambda A = (A^\mathbb N, \cdot^\lambda)$ associated to any
$\mathrm{k}$-algebra $A$ is in fact the cofree $\lambda$-weighted
differential algebra on $A$. To be precise about this, we consider the
category $\mathbf{Dif}_\lambda$ of $\lambda$-weighted differential
$\mathrm{k}$-algebras; as in the introduction, the objects of this
category are $\mathrm{k}$-algebras equipped with a $\mathrm{k}$-linear
endomorphism $\partial$ satisfying~\eqref{deriv}, while the morphisms
are maps of $\mathrm{k}$-algebras preserving $\partial$.

\begin{proposition}
  \label{prop:2}
  For any $\mathrm{k}$-algebra $A$, the operator
  $\partial \colon A^\mathbb N \rightarrow A^\mathbb N$ with
  $(\partial a)_n = a_{n+1}$ makes $G_\lambda A$ into a
  $\lambda$-weighted differential algebra. The algebra morphism
  $G_\lambda A \rightarrow A$ given by taking $0$th components
  exhibits $(G_\lambda A, \partial)$ as the value at $A$ of a right
  adjoint $R$ to the forgetful functor
  $U \colon \mathbf{Dif}_\lambda \rightarrow
  \mathrm{k}\text-\mathbf{Alg}$.
  The adjunction $U \dashv R$ is comonadic.
\end{proposition}

\begin{proof}
  This is~\cite[Propositions~2.7 and 2.8]{GKZ2008} and \cite[Theorem~3.5]{GKZ2014}.
\end{proof}

There is a corresponding (well-known) result for the algebra
$HA = (A^\mathbb N, \text{pointwise})$. Writing
$\mathrm{k}[x]\text-\mathbf{Alg}$ for the category of
$\mathrm{k}$-algebras equipped with an algebra endomorphism, we have:

\begin{proposition}
  \label{prop:3}
  For any $\mathrm{k}$-algebra $A$, the operator
  $\partial \colon A^\mathbb N \rightarrow A^\mathbb N$ with
  $(\partial a)_n = a_{n+1}$ is a ring endomorphism of $HA$. The
  algebra morphism $HA \rightarrow A$ given by taking $0$th component
  exhibits $(HA, \partial)$ as the value at $A$ of a right adjoint $S$
  to the forgetful functor
  $V \colon \mathrm{k}[x]\text-\mathbf{Alg} \rightarrow
  \mathrm{k}\text-\mathbf{Alg}$.
  The adjunction $V \dashv S$ is comonadic.
\end{proposition}

\begin{proof}
  If $(B,\varphi \colon B \rightarrow B) \in
  \mathrm{k}[x]\text-\mathbf{Alg}$
  and $f \colon V(B, \varphi) \rightarrow A$ is a map of
  $\mathrm{k}$-algebras, then the corresponding map
  $\bar f \colon (B, \varphi) \rightarrow (HA, \partial)$ is defined
  by $\bar f(b)_n = f(\varphi^n(b))$. This gives adjointness;
  comonadicity follows as $V$ preserves colimits and is conservative.
\end{proof}

We may now strengthen Proposition~\ref{prop:1} so as to incorporate the
induced comonad structures on $G_\lambda = UR$ and $H = VS$. 

\begin{proposition}
  \label{prop:4}
  The algebra map
  $\gamma = \gamma_\infty \colon G_\lambda A \rightarrow HA$ of
  Proposition~\ref{prop:1}(iii) is the component at $A$ of a comonad
  morphism $G_\lambda \rightarrow H$.
\end{proposition}

\begin{proof}
  If the $\mathrm{k}$-algebra $A$ bears a $\lambda$-weighted
  differential $\partial$, then we have a ring endomorphism of $A$
  given by the operator $\varphi = 1 + \lambda\partial$. Indeed,
  $\mathrm{k}$-linearity and preservation of the unit are clear; as
  for multiplication, we have $\varphi(a)\varphi(b)$ equal to
  \begin{equation*}
    (a + \lambda(\partial a))(b + \lambda(\partial b)) = ab + \lambda((\partial a)b + a(\partial b) + \lambda (\partial a)(\partial b)) = ab + \lambda \partial(ab) = \varphi(ab)
  \end{equation*}
  as required. The assignation
  $(A, \partial) \mapsto (A, 1 + \lambda \partial)$ is thus the action
  on objects of a functor
  $F \colon \mathbf{Dif}_\lambda \rightarrow
  \mathrm{k}[x]\text-\mathbf{Alg}$
  commuting with the forgetful functors to
  $\mathrm{k}\text-\mathbf{Alg}$. Any such functor induces a comonad
  morphism
  $G_\lambda \rightarrow H$---see~\cite[Lemma~4.5.1]{Borceux1994}, for
  example---which in this case is obtained as follows. Take the cofree
  $\lambda$-differential algebra $(G_\lambda A, \partial)$, the induced
  $\mathrm{k}[x]$-algebra $(G_\lambda A, 1 + \lambda\partial)$, and the
  $0$th component homomorphism
  $\varepsilon \colon V(G_\lambda A, 1 + \lambda\partial) \rightarrow
  A$.
  The comonad morphism in question now has its $A$-component given by
  the map of $\mathrm{k}$-algebras underlying
  $\bar \varepsilon \colon (G_\lambda A, 1 + \lambda \partial)
  \rightarrow (HA, \partial)$.
  From Proposition~\ref{prop:3} above, we have that
  \begin{equation*}
    \bar \varepsilon(a)_n = \varepsilon(1 + \lambda \partial)^{\circ n}(a) = \sum_{m} \binom{n}{m} \lambda^m \varepsilon \partial^m(a) = \sum_{m} \binom{n}{m} \lambda^m a_m = \gamma(a)_n
  \end{equation*}
  since the operators $1$ and $\lambda \partial$ commute; so
  $\gamma = \bar \varepsilon$ is the component at $A$ of a comonad
  morphism, as claimed.
\end{proof}

Recall that a comonad $P$ on a category $\CC$ is said to be
\emph{cofree} on a copointed endofunctor
$(T, \varepsilon \colon T \Rightarrow \mathrm{id})$ of $\CC$ if $P$ is
the value at $(T,\varepsilon)$ of a right adjoint to the forgetful
functor $\mathbf{Cmd}(\CC) \rightarrow [\CC, \CC] / \mathrm{id}$ from
comonads to copointed endofunctors.

\begin{proposition}
  \label{prop:6}
  The comonads $G_\lambda$ and $H$ are cofree on copointed
  endofunctors, and the comonad map
  $\gamma \colon G_\lambda \rightarrow H$ is cofree on a map of
  copointed endofunctors.
\end{proposition}

\begin{proof}
  For $G_\lambda$, consider the copointed endofunctor $(S, \sigma)$
  with $SA = (A^2, \cdot^\lambda)$ and $\sigma$ given by the first
  projection. To endow a $\mathrm{k}$-algebra $A$ with a homomorphism
  $a \colon A \rightarrow SA$ satisfying $\sigma a = 1_A$ is easily
  the same as endowing it with a $\lambda$-weighted differential;
  whence the category $(S, \sigma)\text-\mathbf{Coalg}$ of coalgebras
  for this copointed endofunctor is isomorphic over
  $\mathrm{k}\text-\mathbf{Alg}$ to the category
  $\mathbf{Dif}_\lambda$. It follows
  by~\cite[Proposition~22.2]{Kellytransfinite} that the comonad
  $G_\lambda$ induced by the adjunction
  $R \vdash U \colon \mathbf{Dif}_\lambda \rightarrow
  \mathrm{k}\text-\mathbf{Alg}$
  is the cofree comonad on $(S, \sigma)$. The same argument pertains
  for $H$ on considering the copointed endofunctor $(T, \tau)$ with
  $TA = (A^2, \text{pointwise})$ and $\tau$ given again by the first
  projection. Finally, the maps
  $\gamma_2 \colon (A^2, \cdot^\lambda) \rightarrow (A^2,
  \text{pointwise})$
  of Proposition~\ref{prop:1}(iii) are the components of a pointed
  endofunctor map $(S, \sigma) \rightarrow (T, \tau)$, composition
  with which induces the functor
  $F \colon \mathbf{Dif}_\lambda \rightarrow
  \mathrm{k}[x]\text-\mathbf{Alg}$
  of the preceding proof; whence
  $\gamma \colon G_\lambda \rightarrow H$ is induced as the cofree
  comonad morphism on $\gamma_2$.
\end{proof}

\section{Coalgebraic aspects}
\label{sec:coalgebraic-aspects}

As anticipated in the introduction, we may understand the
constructions of the preceding section more straightforwardly using
convolution. Recall that, if $(C, \varepsilon, \delta)$ is a
$\mathrm{k}$-coalgebra and $(A, \eta, \mu)$ is a $\mathrm{k}$-algebra,
the $\mathrm{k}$-linear hom $[C,A]$ becomes a $\mathrm{k}$-algebra
$[C,A]$ under \emph{convolution}, with unit
$e = \eta \varepsilon \colon C \rightarrow \mathrm{k} \rightarrow A$
and product $f \ast g = \mu(f \otimes g)\delta$.

As a first application, we consider the coalgebra $C({\lambda})$ whose
underlying $\mathrm{k}$-module is free on $\{e,d\}$ with counit and
comultiplication defined by:
\begin{equation*}
  \varepsilon(e)= 1 \ , \ \varepsilon(d)=0\ , \ \delta(e)=e\otimes e \ , \ \delta(d)=d\otimes e+e\otimes d+\lambda \ d\otimes d \ .
\end{equation*}
On the other hand, we have the coalgebra $D$ with the same underlying
$\mathrm{k}$-module but the ``set-like'' coalgebra structure given by
\begin{equation*}
  \varepsilon(e)=\varepsilon(d)=1\ , \ \delta(e)=e\otimes e \ , \ \delta(d)=d\otimes d \ .
\end{equation*}
Moreover, there is a coalgebra morphism $\xi \dd D\to C({\lambda})$
with $\xi(e)=e$ and $\xi(d)=\lambda d+e$. It is now direct from the
definitions that:

\begin{proposition}\label{Hurwitz2}
  For any $\mathrm{k}$-algebra $A$, there are isomorphisms
  $[C(\lambda), A] \cong (A^2, \cdot^\lambda)$ and
  $[D, A] \cong (A^2, \text{pointwise})$, and modulo these,
  $[\xi, A] = \gamma_2 \colon (A^2, \cdot^\lambda) \rightarrow (A^2,
  \text{pointwise})$.
\end{proposition}

In fact, the coalgebras $C(\lambda)$ and $D$ are pointed by the maps
$\eta \colon \mathrm{k} \rightarrow C(\lambda)$ and
$\eta \colon \mathrm{k} \rightarrow D$ with $\eta(1) = e$, and
$\xi \colon D \rightarrow C(\lambda)$ is a map of pointed coalgebras.
This structure transports under convolution to make $\gamma_2$ into
the map of \emph{copointed} endofunctors
$(S, \sigma) \rightarrow (T, \tau)$ of Proposition~\ref{prop:6}. We
saw in that Proposition that the comonads $G_\lambda$ and $H$ and
comonad morphism $\gamma \colon G_\lambda \rightarrow H$ may be
derived from these data using cofreeness; our next result reconstructs
this purely in the world of coalgebras.

We first recall the construction of the free $\mathrm{k}$-bialgebra on
a pointed $\mathrm{k}$-coalgebra $(E, \eta)$. For each $\ell > 1$,
define the coalgebra $E_{\ell}$ as the joint coequaliser
$\gamma \dd E^{\otimes \ell} \to E_{\ell}$ of the $\ell$ coalgebra
morphisms
\begin{equation*}
  \eta\otimes 1\otimes \dots \otimes 1 \ , \ 1 \otimes \eta\otimes \dots \otimes 1 \ , \ \dots \ , \ 1\otimes 1\otimes \dots \otimes \eta \dd E^{\otimes (\ell -1)}\to E^{\otimes \ell}\ .
\end{equation*}
Of course, we also can put $E_0=\mathrm{k}$ and $E_1=E$. Then for each
$\ell \geqslant 0$ we have a unique coalgebra morphism
$\zeta\dd E_{\ell}\to E_{\ell +1}$ such that the square
\begin{equation*}
  \xymatrix{
    E^{\otimes \ell} \ar[rr]^-{\gamma} \ar[d]_-{1 \otimes \dots \otimes 1 \otimes \eta} && E_{\ell} \ar[d]^-{\zeta} \\
    E^{\otimes (\ell+1)} \ar[rr]^-{\gamma} && E_{\ell+1}
  }
\end{equation*}
commutes. Now the free monoid construction in \cite{Dubuc1974} shows
that the free bialgebra $E_\infty$ on the pointed coalgebra $E$ is
obtained as the colimit of the chain
\begin{equation*}
  E_{0}\xrightarrow{\zeta=\eta} E_{1}\xrightarrow{\zeta} E_{2} \xrightarrow{\zeta} \dots \xrightarrow{\zeta} E_{\ell} \xrightarrow{\zeta} E_{\ell+1} \xrightarrow{\zeta} \dots \rlap{ .}
\end{equation*}
We now apply this to the coalgebras $D$ and $C(\lambda)$. As
$\mathrm{k}$-modules, $D^{\otimes \ell}$ and
$C(\lambda)^{\otimes \ell}$ are both free on the basis
$\{e,d\}^{\ell}$, which we think of as the set of words $W$ of length
$\ell$ in the letters $e$ and $d$; while $D_{\ell}$ and
$C(\lambda)_{\ell}$ are obtained by quotienting out by the relation
$WeW' \sim eWW'$ on basis words. They are thus vector spaces of
dimension $\ell+1$ with basis elements $e^rd^s$ where $r+s = \ell$.
Omitting to write the $e^r$ term, we can thus use the isomorphic basis
$\{d_s\mid 0\leqslant s\leqslant \ell\}$.

As a coalgebra $D_{\ell}$, like $D$, is ``set-like'' with respect to
these basis vectors: that is $\varepsilon(d_s)= 1$ and
$\delta(d_s)=d_s\otimes d_s$ for each $0 \leqslant s \leqslant \ell$.
On the other hand:

\begin{proposition}\label{coalgC}
  The coalgebra $C(\lambda)_{\ell}$, with respect to its basis
  $\{d_s\mid 0\leqslant s\leqslant \ell\}$, has
  $\varepsilon(d_{0})= 1$, $\varepsilon(d_{n})= 0$ for $0< n\le\ell$,
  and for $0\le n\le \ell$
  \begin{equation*}
    \delta(d_n)=\sum_{n=r+s+t}{\binom{n}{r,s,t}\lambda^td_{r+t}\otimes d_{s+t}} \ .
  \end{equation*}
\end{proposition}

\begin{proof}
  In the coalgebra $C(\lambda)^{\otimes \ell}$, we have for all
  $m+n = \ell$ that
  \begin{equation*}
    \delta(e^md^n) = (e\otimes e)^m(d\otimes e+e\otimes d +\lambda
    d\otimes d)^n\rlap{ .}
  \end{equation*}
  We cannot expand binomially since $d \otimes e$ and $e \otimes d$ and
  $\lambda d \otimes d$ do not commute in
  $C(\lambda)^\ell \otimes C(\lambda)^\ell$; but they do after applying
  $\gamma \otimes \gamma \colon C(\lambda)^\ell \otimes C(\lambda)^\ell
  \rightarrow C(\lambda)_\ell \otimes C(\lambda)_\ell$,
  and so we find that
  \begin{eqnarray*}
    (\gamma \otimes \gamma)\delta(e^md^n) &=&
    \sum_{n=r+s+t}{\binom{n}{r,s,t}(e\otimes e)^m(d\otimes e)^r(e\otimes d)^s\lambda^t(d\otimes d)^t} \\
    &=& \sum_{n=r+s+t}{\binom{n}{r,s,t}\lambda^t(e^{m+s}d^{r+t}\otimes e^{m+r}d^{s+t})} \ .
  \end{eqnarray*}
  This implies the result for comultiplication; the counit case is left
  as an exercise.
\end{proof}

Moreover, the pointed coalgebra morphism
$\xi \colon D \rightarrow C(\lambda)$ induces for each
$\ell \in \mathbb N$ a coalgebra morphism
$\xi_\ell \colon D_\ell \rightarrow C(\lambda)_\ell$, unique such that
the square
\begin{equation*}
  \xymatrix{
    D^{\otimes \ell} \ar[rr]^-{\gamma} \ar[d]_-{\xi^{\otimes \ell}} && D_{\ell} \ar[d]^-{\xi_\ell} \\
    C(\lambda)^{\otimes (\ell)} \ar[rr]^-{\gamma} && C(\lambda)_{\ell+1}
  }
\end{equation*}
commutes; passing to the colimit, we obtain a map of
$\mathrm{k}$-bialgebras
$\xi_\infty \colon D_\infty \rightarrow C(\lambda)_\infty$. Arguing as
in the preceding proof, we find for finite $\ell$ that
\begin{equation*}
  \xi_\ell(d_n) = \gamma\xi^{\otimes \ell}(e^k d^n) = e^k (\lambda d + e)^n = e^k \sum_m
  \binom{n}{m} \lambda^m d^m e^{n-m} = \sum_m \lambda^m \binom{n}{m} d_m
\end{equation*}
which now yields the following result generalising
Proposition~\ref{Hurwitz2}. 

\begin{proposition}
  For any $\mathrm{k}$-algebra $A$ and
  $\ell \in \mathbb N \cup \{\infty\}$, there are isomorphisms
  $[C(\lambda)_\ell, A] \cong (A^{\ell+1}, \cdot^\lambda)$ and
  $[D_\ell, A] \cong (A^{\ell+1}, \text{pointwise})$; modulo these
  $[\xi_\ell, A] = \gamma_{\ell+1}$.
\end{proposition}

\begin{proof}
  For finite $\ell$, we simply compare the preceding formulae with
  Proposition~\ref{prop:1}; for $\ell = \infty$, we observe that
  convolution
  $[\thg, A] \colon \mathrm{k}\text-\mathbf{Coalg} \rightarrow
  \mathrm{k}\text-\mathbf{Alg}^\mathrm{op}$ preserves colimits.
\end{proof}

The functor
$\mathrm{k}\text-\mathbf{Coalg} \rightarrow
[\mathrm{k}\text-\mathbf{Alg},
\mathrm{k}\text-\mathbf{Alg}]^\mathrm{op}$
sending $C$ to $[C, \thg]$ carries tensor product of coalgebras to
composition of endofunctors, and so carries each
$\mathrm{k}$-bialgebra $C$ to a comonad $[C,\thg]$ on
$\mathrm{k}\text-\mathbf{Alg}$. Of course, for the bialgebras
$C(\lambda)_\infty$ and $D_\infty$, the associated comonads are
$G_\lambda$ and $H$; this follows from Proposition~\ref{prop:6} and
the fact that the convolution functor
$\mathrm{k}\text-\mathbf{Coalg} \rightarrow
[\mathrm{k}\text-\mathbf{Alg},
\mathrm{k}\text-\mathbf{Alg}]^\mathrm{op}$
sends each free bialgebra sequence in $\mathrm{k}\text-\mathbf{Coalg}$
to a cofree comonad sequence in
$[\mathrm{k}\text-\mathbf{Alg}, \mathrm{k}\text-\mathbf{Alg}]$.

In light of these investigations, we may wonder whether there are
other kinds of $\mathrm{k}$-linear ``derivation'' which a
$\mathrm{k}$-algebra can bear, satisfying some different kind of
``Leibniz identity''. Our next result denies this possibility.

\begin{proposition}\label{pauc} 
  All pointed $\mathrm{k}$-coalgebras whose underlying module is free
  of rank $2$ are isomorphic to $C({\lambda})$ for some $\lambda$.
\end{proposition}

\begin{proof}
  Suppose $(C', \eta \colon \mathrm{k} \rightarrow C')$ is a pointed
  coalgebra with basis $\{e', d'\}$ such that $\eta(1)=e'$. Since
  $\eta$ is a coalgebra morphism, $\varepsilon(e')=1$ and
  $\delta(e')=e'\otimes e'$. Suppose that $\varepsilon(d')=\gamma$; by
  making the change of basis $e=e'$ and $d=d'-\gamma e$ we have that
  \begin{equation*}
    \varepsilon(e)= 1 \ , \ \varepsilon(d)=0\ , \ \delta(e)=e\otimes e \ .
  \end{equation*}
  Suppose
  $\delta(d) = \rho e\otimes e + \sigma e\otimes d + \tau d\otimes e +
  \upsilon d\otimes d$.
  From the counit properties $(\varepsilon\otimes 1)\delta = 1$ and
  $(1\otimes \varepsilon)\delta = 1$, we deduce
  $d=\rho e + \sigma d = \rho e + \tau d$ and so $\sigma = \tau = 1$ and
  $\rho = 0$. It is easily checked that the coassociativity condition is
  now automatic. So we have our result with $\upsilon = \lambda$.
\end{proof}

\section{The $\Lambda$-weighted tensor product of species}
\label{sec:lambda-weight-tens}

In~\cite{wtp} is described a ``categorification'' of the
$\lambda$-Hurwitz product on $\mathrm{k}$-algebras; in the rest of
this paper, we give corresponding ``categorifications'' of the results
of the previous sections. To give these generalisations, we replace
the commutative $\mathbb{Q}$-algebra $\mathrm{k}$ with a complete and
cocomplete symmetric monoidal closed category $\CV$, replace
$\mathrm{k}$-modules with what we shall call \emph{$\CV$-vector
  spaces}---$\CV$-categories admitting finite coproducts and
\emph{$\CV$-tensors}~\cite[\S 3.7]{KellyBook}---and replace
$\mathrm{k}$-algebras by \emph{$\CV$-algebras}, that is, $\CV$-vector
spaces with a monoidal structure which preserves finite coproducts and
$\CV$-tensors in each variable separately. We call a $\CV$-algebra
\emph{symmetric} when its monoidal structure is so\footnote{Although now 
this is extra structure not just a condition}.

Throughout the rest of this section, we fix some $\Lambda \in \CV$.
For any (symmetric) $\CV$-algebra $\CA$, there is a now (symmetric)
$\CV$-algebra structure on $\CA^\mathbb N$ with unit
$J = (I, 0, \dots)$ and binary tensor $\ast_\Lambda$ given by
\begin{equation}\label{Nten}
  (M \ast_{\Lambda} N)_n = \sum_{n = r + s + t}{\Lambda^{\otimes r} \otimes M_s\otimes N_t} \rlap{ .}
\end{equation} 
In~\cite{wtp} was described a similar tensor product on the
$\CV$-category $\CA^\mathfrak{S}$ of \emph{$\CA$-valued
  species}---here $\mathfrak S$ is the groupoid of finite sets and
bijections---whose unit $J$ and binary tensor $\ast_\Lambda$ are given
by
\begin{equation}\label{Lten}
  JX =
  \begin{cases}
    I & \text{if $X = \emptyset$;}\\
    0 & \text{otherwise;}
  \end{cases} \quad \text{and} \quad 
  (M*_{\Lambda}N)X = \sum_{\substack{U, V \subset X\\X=U\cup V}}{\Lambda^{\abs{U\cap V}} \otimes MU\otimes NV} \rlap{ .}
\end{equation}
As previously, these monoidal structures restrict back to the
respective subcategories $\CA^\ell \subset \CA^\mathbb N$ and
$\CA^{\mathfrak S_{\ell}} \subset \CA^{\mathfrak S}$ for any
$\ell \in \mathbb N$; here $\mathfrak S_{\ell}$ is the full
subcategory of $\mathfrak S$ on sets of cardinality $<\ell$. Our next
result will reconstruct these monoidal structures through an argument
like that of Section~\ref{sec:coalgebraic-aspects}.

We exploit the symmetric monoidal bicategory $\CV\text-\mathbf{Vect}$
of $\CV$-vector spaces, wherein $1$-cells are \emph{$\CV$-linear}
$\CV$-functors---ones preserving finite coproducts and
$\CV$-tensors---$2$-cells are $\CV$-natural transformations, and the
tensor product $\otimes$ classifies \emph{$\CV$-bilinear}
$\CV$-bifunctors---ones which preserve finite coproducts and
$\CV$-tensors in each variable separately. The unit object is $\CV$
itself. This monoidal bicategory is biclosed in the sense
of~\cite{mbaHa}, with the internal hom $[\CA, \CB]$ being the functor
$\CV$-category of $\CV$-linear $\CV$-functors. There is a
free-forgetful biadjunction
$\CV\text-\mathbf{Vect} \leftrightarrows \mathbf{Cat}$ whose left
adjoint is \emph{strong monoidal} with respect to the tensor product
of $\CV$-vector spaces and the cartesian product of categories. For a
given category $\CA$, we write $\spn{\CA}$ for the free $\CV$-vector
space thereon.

A (symmetric) monoidale\footnote{Also called a \emph{pseudomonoid}.} in
$\CV\text-\mathbf{Vect}$ is precisely a (symmetric) $\CV$-algebra in
the sense described above; correspondingly, we refer to comonoidales
in $\CV\text-\mathbf{Vect}$ as \emph{$\CV$-coalgebras}. By convolution
as in~\cite{mbaHa}, the internal hom from a comonoidale to a monoidale
in a biclosed monoidal bicategory is again a monoidale; so if $\CC$ is
a $\CV$-coalgebra and $\CA$ is a $\CV$-algebra, then the $\CV$-linear
hom $[\CC, \CA]_\CV$ is a $\CV$-algebra.

Consider now the $\CV$-coalgebra $\CC(\Lambda)$ with underlying
$\CV$-vector space $\spn{E, D}$---so that
$\CC(\Lambda) \otimes \CC(\Lambda) \simeq \spn{E \otimes E, E \otimes
  D, D \otimes E, D \otimes D}$---with
counit $\varepsilon \colon \CC(\Lambda) \rightarrow \CV$ and
comultiplication
$\delta \colon \CC(\Lambda) \rightarrow \CC(\Lambda) \otimes
\CC(\Lambda)$ given on generators by
\begin{equation}\label{eq:4}
  \varepsilon(E)= I \ , \ \varepsilon(D)=0\ , \ \delta(E)=E\otimes E \ , \ \delta(D)=D\otimes E+E\otimes D+\Lambda \otimes D\otimes D\rlap{.}
\end{equation}
The coassociativity and counit coherences are given on generators in
the obvious way. There is a comonoidale morphism
$\CV \rightarrow \CC(\Lambda)$ sending $V$ to $V \otimes E$ so making
$\CC(\Lambda)$ into a \emph{pointed} $\CV$-coalgebra. Convolving with
a $\CV$-algebra $\CA$ gives monoidal structures on the $\CV$-linear
hom $[\CC(\Lambda), \CA]_\CV \simeq \CA^2$; the unit object is
$J = (I, 0)$, while the binary tensor is given by
\begin{equation*}
  (M_0, M_1) \ast_{\Lambda} (N_0, N_1) = (M_0 \otimes N_0, M_0 \otimes N_1 + M_1 \otimes N_0 + \Lambda \otimes M_1 \otimes N_1)
\end{equation*}
so that we re-find the case $\ell = 2$ of the $\Lambda$-weighted
tensor product~\eqref{Nten}, which is also the case $\ell=2$
of~\eqref{Lten}. We now show how to obtain the corresponding tensor
products on $\CA^\mathbb N$ or $[\mathfrak S, \CA]$ from this by
arguing as in Section~\ref{sec:coalgebraic-aspects}.

First, since the free $\CV$-vector space $2$-functor
$\mathbf{Cat} \rightarrow \CV\text-\mathbf{Vect}$ is strong symmetric
monoidal and cocontinuous, it preserves the construction of free
(symmetric) monoidales. As the free monoidal category and the free
symmetric monoidal category on the pointed category
$\{E\} \rightarrow \{E, D\}$ are $(\mathbb N, +, 0)$ and
$(\mathfrak S, +, 0)$ respectively, we conclude that:

\begin{proposition}\label{prop:5}
  The free $\CV$-algebra and free symmetric $\CV$-algebra on the
  pointed $\CV$-vector space $\spn{E} \rightarrow \spn{E,D}$ are
  respectively $\spn{\mathbb N}$ and $\spn{\mathfrak S}$ under the
  monoidal structure given on basis elements by disjoint union.
\end{proposition}

The symmetric monoidal structure on the bicategory
$\CV\text-\mathbf{Vect}$ lifts to the bicategories
$\CV\text-\mathbf{Alg}$ and $\CV / \CV\text-\mathbf{Vect}$ of
$\CV$-algebras and of pointed $\CV$-vector spaces; when endowed with
these monoidal structures, the forgetful
$\CV\text-\mathbf{Alg} \rightarrow \CV / \CV\text-\mathbf{Vect}$ is
thus strict monoidal, so that its left biadjoint is an opmonoidal
homomorphism. It follows that the biadjunction passes to the
respective bicategories of comonoidales:
\begin{equation*}
  \CV\text-\mathbf{Bialg} \leftrightarrows \CV / \CV\text-\mathbf{Coalg}\rlap{ .}
\end{equation*}

Explicitly, this means that, if $Z \colon \CV \rightarrow \CC$ is a
pointed $\CV$-coalgebra, and the map
$\iota \colon \CC \rightarrow \CC_\infty$ of pointed objects exhibits
$\CC_\infty$ as the free $\CV$-algebra on
$Z \colon \CV \rightarrow \CC$ seen as a pointed $\CV$-vector space,
then $\CC_\infty$ bears a coalgebra structure making it into the free
$\CV$-bialgebra on $\CV \rightarrow \CC$, with counit and
comultiplication $\CV$-functors obtained as the essentially-unique
homomorphisms of $\CV$-algebras rendering the following diagram
commutative to within natural isomorphisms:
\begin{equation}\label{eq:1}
  \xymatrix{
    & \CC \ar[r]^-{\delta} \ar[d]^-{\iota} \ar[ld]_-{\varepsilon} & \CC\otimes \CC \ar[d]^-{\iota \otimes \iota} \\
    \CV & \CC_{\infty} \ar[r]_-{\delta} \ar[l]^-{\varepsilon} & \CC_{\infty}\otimes \CC_{\infty}\rlap{ .}
  }
\end{equation}
Of course, if we construct instead the free \emph{symmetric}
$\CV$-algebra on $\CC$, then the above argument show that we in fact
obtain the free \emph{symmetric} $\CV$-bialgebra on $\CC$. Applying
these two constructions to the pointed $\CV$-coalgebra $\CC(\Lambda)$
and using Proposition~\ref{prop:5}, we induce $\CV$-bialgebra
structures on $\spn{\mathbb N}$ and on $\spn{\mathfrak S}$; we will
show that these convolve to give the weighted tensor products
of~\eqref{Nten} and~\eqref{Lten}. The argument in the two cases is
similar, and that for $\spn{\mathbb N}$ is exactly like that in
Section~\ref{sec:coalgebraic-aspects} above; so we go through the
details only for $\spn{\mathfrak S}$.

\begin{theorem}\label{freeLambda}
  The free symmetric $\CV$-bialgebra $\CC(\Lambda)_\infty$ on the
  pointed $\CV$-coalgebra $\CC(\Lambda)$ is $\spn{\mathfrak{S}}$ with
  as algebra structure the $\CV$-linear extension of
  $(\mathfrak S, +, 0)$, and coalgebra structure determined on basis
  elements $X \in \mathfrak S$ by:
  \begin{equation}\label{eq:3}
    \varepsilon(X) =
    \begin{cases}
      I & \text{if $X = 0$;}\\
      0 & \text{otherwise,}
    \end{cases}
    \quad \text{and} \quad \delta(X) = \displaystyle\sum_{X = U \cup V} \Lambda^{\otimes \abs{U \cap V}} \otimes U \otimes V\rlap{ .}
  \end{equation}
  For any $\CV$-algebra $\CA$, the convolution algebra structure on
  $[\CC(\Lambda)_\infty, \CA]_\CV \simeq \CA^\mathfrak S$ is given by
  the $\Lambda$-weighted tensor product of~\eqref{Lten}.
\end{theorem}

\begin{proof}
  \cite[Proposition~23]{wtp} shows that the above data do indeed
  define a $\CV$-bialgebra structure on $\spn{\mathfrak S}$. Let
  $\iota \colon \spn{E, D} \rightarrow \spn{\mathfrak S}$ be given on
  generators by $E \mapsto 0$ and $D \mapsto 1$. By
  Proposition~\ref{prop:5}, this map exhibits $\spn{\mathfrak S}$ as
  the free symmetric $\CV$-algebra on the pointed object
  $\spn{E} \rightarrow \spn{E, D}$; now by comparing~\eqref{eq:3}
  with~\eqref{eq:4}, we see that the following diagram commutes to
  within isomorphism since it does so on generators:
  \begin{equation*}
    \xymatrix{
      & \spn{E, D} \ar[r]^-{\delta} \ar[d]^-{\iota} \ar[ld]_-{\varepsilon} 
      & \spn{E,D} \otimes \spn{E,D} \ar[d]^-{\iota \otimes \iota} \\
      \CV & \CV\mathfrak{S} \ar[r]_-{\delta} \ar[l]^-{\varepsilon} & \CV\mathfrak{S}\otimes \CV\mathfrak{S}}\rlap{ .}
  \end{equation*}
  By the argument following Proposition~\ref{prop:5}, we conclude that
  $\spn{\mathfrak S}$, equipped with the given algebra and coalgebra
  structures, is the free symmetric $\CV$-bialgebra
  $\CC(\Lambda)_\infty$ on $\CC(\Lambda)$. The final claim follows
  immediately by comparing~\eqref{eq:3}
  with~\eqref{Lten}.
\end{proof}

We now relate the $\Lambda$-weighted tensor products on
$\CA^\mathbb N$ and $\CA^\mathfrak S$ to the pointwise ones, following
the pattern set out in Section~\ref{sec:coalgebraic-aspects} above. To
this end, consider the pointed $\CV$-coalgebra $\CD$ with the same
underlying $\CV$-vector space $\spn{E,D}$ as $\CC(\Lambda)$ and the
same pointing, but the diagonal comonoidale structure:
\begin{equation*}
  \varepsilon(E)= I \ , \ \varepsilon(D)=I\ , \ \delta(E)=E\otimes E \ , \ \delta(D)=D\otimes D\rlap{ .}
\end{equation*}
For any $\CV$-algebra $\CA$, convolution with $\CD$ induces the
pointwise $\CV$-algebra structure on the $\CV$-linear hom
$[\CD, \CA] \simeq \CA^2$. Arguing as previously, we now have that:

\begin{theorem}\label{thm:1}
  The free symmetric $\CV$-bialgebra $\CD_\infty$ on the pointed
  $\CV$-coalgebra $\CD$ is $\spn{\mathfrak S}$ with as algebra
  structure the $\CV$-linear extension of
  $(\mathfrak S, +, \emptyset)$, and coalgebra structure given on
  homogeneous elements $X \in \mathfrak S$ by:
  \begin{equation}\label{eq:3}
    \varepsilon(X) = I
    \quad \text{and} \quad \delta(X) = X \otimes X\rlap{ .}
  \end{equation}
  For any $\CV$-algebra $\CA$, the convolution algebra structure on
  $[\CD_\infty, \CA]_\CV \simeq \CA^\mathfrak S$ is the pointwise
  algebra structure.
\end{theorem}

To compare the pointwise and the $\Lambda$-weighted monoidal
structures on $\CA^\mathfrak S$, it will thus suffice to compare their
restrictions to $\CA^2$. So consider the $\CV$-linear $\CV$-functor
$\Theta \colon \CD \rightarrow \CC(\Lambda)$ defined on generators by
$\Theta(E) = E$ and $\Theta(D) = \Lambda \otimes D + E$; we may without
difficulty equip this with the structure of a strong morphism of
pointed comonoidales in $\CV\text-\mathbf{Vect}$. The pointings on
$\CC(\Lambda)$ and $\CD$ correspond to the $\CV$-algebra morphisms
given by the first projection, while
$\Theta \colon \CD \rightarrow \CC(\Lambda)$ transports under convolution
to yield $\hat \Theta \colon \CA^2 \rightarrow \CA^2$ sending
$(M_0, M_1)$ to $(M_0, M_0 + \Lambda \otimes M_1)$. Since $\Theta$ is a
strong morphism of pointed comonoidales, this yields:

\begin{proposition}\label{smfover}
  $\widehat{\Theta} \dd (\CA^2, *_{\Lambda}, J) \to (\CA^2,
  \text{pointwise})$ is strong monoidal.
\end{proposition}

The map $\Theta$ induces a map
$\Theta_\infty \colon \CD_\infty \rightarrow \CC(\Lambda)_\infty$ of
symmetric bimonoidales which is determined in an essentially-unique
manner by the requirement that the square
\begin{equation}\label{Thetabar}
  \begin{aligned}
    \xymatrix{
      \CD \ar[rr]^-{\Theta} \ar[d]_-{\iota} && \CC(\Lambda) \ar[d]^-{\iota} \\
      \CD_\infty \ar[rr]^{\Theta_\infty} && \CC(\Lambda)_\infty}
  \end{aligned}
\end{equation}
should commute to within isomorphism. It is easy to see that these
requirements are satisfied by taking $\Theta_\infty$ to be defined on
basis elements $X \in \mathfrak S$ by:
\begin{equation*}
  \Theta_\infty(X) = \sum_{W \subseteq X} \Lambda^{\otimes \abs W} \otimes W\rlap{ .}
\end{equation*}
The proof that this is a map of bimonoidales uses the equation
$U+V=U\cup V+U\cap V$ for $U, V\subseteq X$. For any $\CV$-algebra
$\CA$, convolution with $\Theta_\infty$ yields a $\CV$-functor
$\widehat{\Theta}_\infty \colon \CA^\mathfrak S \rightarrow \CA^\mathfrak
S$
defined by
$\bar{\Theta}_\infty(M)X = \sum_{W \subseteq X} \Lambda^{\otimes \abs W}
\otimes MW$,
and since $\Theta_\infty$ is a strong map of bimonoidales, we conclude
that:
\begin{proposition}\label{smfover}
  $\widehat{\Theta}_\infty \dd (\CA^\mathfrak S, *_{\Lambda}, J) \to
  (\CA^\mathfrak S, \text{pointwise})$ is strong
  monoidal.
\end{proposition}

Note that the preceding two results are further examples of a
transformation converting a convolution product into pointwise product
as promoted in \cite{55}.

\section{Correspondences of Dold--Kan type}
\label{sec:corr-dold-kan}
In these final two sections, we consider a different categorification
of Sections~\ref{wHp}--~\ref{sec:coalgebraic-aspects}. This time, we
replace modules over a commutative ring $\mathrm{k}$ by \emph{additive
  Karoubian} categories---$\mathbf{Ab}$-enriched categories which
admit finite biproducts and splittings of idempotents; and we replace
(commutative) $\mathrm{k}$-algebras by (symmetric)
\emph{$\mathbf{Ab}$-algebras}: (symmetric) monoidal additive Karoubian
categories.

In~\cite{Lack2015Combinatorial} is described a general theory for
establishing equivalences of additive Karoubian categories. From an
ordinary category $\CP$ equipped with a subcategory of monomorphisms
$\CM$ satisfying some axiomatic assumptions is constructed a category
$\CD$ enriched over the category of pointed sets together with an
equivalence
\begin{equation}\label{eq:2}
  \Gamma \colon [\CD, \CX]_{\mathrm{pt}} \rightarrow [\CP, \CX]  
\end{equation}
for each additive Karoubian $\CX$; here, on the right we have the
ordinary functor category, and on the left the category of zero-map
preserving functors. For suitable choices of $\CP$ and $\CM$,
equivalences obtained in this way include the classical
Dold--Puppe--Kan correspondence~\cite{Dold1961Homologie} between
simplicial abelian groups and chain complexes; the correspondence
between cubical and semi-simplicial abelian groups; and the
equivalence between linear species and the
\emph{$\mathrm{FI}\sharp$-modules} of~\cite{Church2015FI-modules}.

In the situation of~\eqref{eq:2}, if the additive Karoubian $\CX$ is
an $\mathbf{Ab}$-algebra, then so too is $[\CP, \CX]$ under the
pointwise tensor product. Transporting across the equivalence yields
an $\mathbf{Ab}$-algebra structure also on~$[\CD, \CX]_\mathrm{pt}$;
in fact, this turns out to be a Hurwitz-style ($\lambda = 1$) tensor
product, and the comparison functor~\eqref{eq:2} yet another example
of a transform taking convolution to pointwise product.

We show this in detail for a special but useful case of~\eqref{eq:2}.
Let $\CC$ be a category equipped with an orthogonal factorisation
system $(\CE, \CM)$ in the sense of~\cite{FreydKelly} such that all
$\CM$-maps are monomorphisms, pullbacks of $\CM$-maps along arbitrary
morphisms exist, and each $A \in \CC$ has only a finite set of
distinct $\CM$-subobjects. We take $\CP = \mathbf{Par}(\CC, \CM)$,
whose objects are those of $\CC$, whose maps from $A$ to $B$ are
isomorphism-classes of spans
$m \colon A \leftarrowtail R \rightarrow B \colon f$ in $\CC$ with
$m \in \CM$, and whose composition is by pullback. We write $\CR$ for
the category whose objects are those of $\CC$ and whose maps are the
$\CE$-maps, and we define $\CD$ to be the free category with zero maps
on $\CR$. With this choice of $\CP$ and $\CD$ we obtain
by~\cite[Example~3.1]{Lack2015Combinatorial} our first instance of an
equivalence~\eqref{eq:2}, which, since $\CD$ is free on $\CR$, may be
written more simply as
\begin{equation}\label{eq:6}
  \Gamma \colon [\CR, \CX] \rightarrow [\CP, \CX]\rlap{ .}
\end{equation}

In order to see how the pointwise tensor product on $[\CP, \CX]$
transports under this equivalence, we will need an explicit
description of both $\Gamma$ and its pseudoinverse. Choose for each
$A \in \CC$ a representing set $\mathrm{Sub}(A)$ of $\CM$-subobjects,
and write $B \leqslant_n A$ to mean that
$n \colon B \rightarrowtail A$ is in $\mathrm{Sub}(A)$ and $B <_n A$
to mean that $n \in \mathrm{Sub}(A)$ is \emph{proper}: that is,
non-invertible. For $F \in [\CR, \CX]$, we now take
$\Gamma F \colon \CP \rightarrow \CX$ to be given on objects by the
(finite) direct sum
\begin{equation*}
  (\Gamma F)A = \bigoplus_{B \leqslant_n A} FB\rlap{ .}
\end{equation*}
To specify $\Gamma F$ at a map
$m \colon A \leftarrowtail R \rightarrow A' \colon f$ in $\CP$,
suppose that $B \leqslant_n A$ and that $B' \leqslant_{n'} A'$, and
define $\xi(m,f)_{nn'} \colon FB \rightarrow FB'$ to be $Fe$ if there is
a (necessarily unique) diagram of the form
\begin{equation}\label{eq:5}
  \vcenter{\hbox{\xymatrix{
        B \ar@{ >->}[d]_-{n} \ar@{=}[r] & B \ar@{ >->}[d]^-{p} \ar@{->}[r]^-{e \in \CE} & B' \ar@{ >->}[d]^-{n'}\\
        A & R \ar@{ >->}[l]_-{m} \ar[r]^-{f} & A' 
      }}}
\end{equation}
and define $\xi(m,f)_{nn'} = 0$ otherwise. We now take
$(\Gamma F)(m,f) \colon (\Gamma F)A \rightarrow (\Gamma F)A'$ to be
the matrix of size $\mathrm{Sub}(A) \times \mathrm{Sub}(A')$ with
entries $\xi(m,f)_{nn'}$. Note in particular that, if $f = m$
in~\eqref{eq:5}, then $e$ must be invertible, whence $n = n'$ in
$\mathrm{Sub}(A)$; thus $\xi(m,m)_{nn'}$ is the identity if $n = n'$
and $n$ factors through $m$, and is zero otherwise. This shows that
$(\Gamma F)(m,m)$ is the idempotent on
$\bigoplus_{B \leqslant_n A} FB$ which projects onto those summands
$n \in \mathrm{Sub}(A)$ which factor through $m$.

The inverse equivalence $N \colon [\CP, \CX] \rightarrow [\CR, \CX]$
to~\eqref{eq:2} sends $H \in [\CP, \CX]$ to
$N H \colon \CR \rightarrow \CX$ defined on objects by
\begin{equation}\label{eq:7}
  (N H)(A) = \bigcap_{R <_m A} \text{ker}(H(m,m) \colon HA \rightarrow HA)\rlap{ ;}
\end{equation}
note that this is well-defined in the Cauchy-complete $\CX$, since the
limit involved may be constructed by splitting the idempotent
$\prod_{R <_m A}(1 - H(m,m))$ of the ring $\CX(HA,HA)$. The action of
$N H$ on a morphism $e \colon A \rightarrow A'$ is the unique
factorisation of the composite
$(N H)(A) \rightarrowtail HA \rightarrow HA'$ through
$(N H)(A') \rightarrowtail HA'$; the existence of such a factorisation
is verified in~\cite[Theorem~4.1]{Lack2015Combinatorial}, while the
fact that $\Gamma$ and $N$ are indeed pseudoinverse is proved by
Theorem~6.7 of \emph{ibid}.

Using the above formulae, we may now derive the existence of a
Hurwitz-style tensor product on $[\CR, \CX]$ for any
$\mathbf{Ab}$-algebra $\CX$ which transports the pointwise one on
$[\CP, \CX]$. In the statement of the following result, we call a pair
of subobjects $n, n' \in \mathrm{Sub}(A)$ \emph{covering} if any
$\CM$-map through which both $n$ and $n'$ factor is invertible.

\begin{proposition}
  \label{prop:7}
  Let $\CC$ be a category equipped with an $(\CE, \CM)$-factorisation
  system, let all pullbacks along $\CM$-maps exist and let each
  $A \in \CC$ have but a finite set of $\CM$-subobjects. Writing as
  above $\CP = \mathbf{Par}(\CC, \CM)$ and $\CR$ for the category of
  $\CE$-maps, there is for any (symmetric) $\mathbf{Ab}$-algebra $\CX$
  a (symmetric) $\mathbf{Ab}$-algebra structure $(\ast, J)$ on
  $[\CR, \CX]$ whose unit $J$ and binary tensor $\ast$ are given by
  \begin{equation}
    J(A) =
    \begin{cases}
      I & \text{if $\abs{\mathrm{Sub}(A)} = 1$;}\\
      0 & \text{otherwise} 
    \end{cases} \quad 
    \text{and} \quad (F \ast G)(A) =\ \vcenter{\hbox{$\displaystyle \bigoplus_{\substack{B \leqslant_n A,\ B' \leqslant_{n'} A \\ n, n' \text{ are covering}}} FB \otimes GB'$}}
    \label{eq:8}
  \end{equation}
  Moreover, \eqref{eq:6} is strong monoidal as a functor
  $([\CR, \CX], \ast) \rightarrow ([\CP, \CX], \text{pointwise})$.
\end{proposition}

\begin{proof}
  It suffices to show that the pointwise tensor product on
  $[\CP, \CX]$ transports to the given structure on $[\CR, \CX]$. For
  the unit this is immediate from~\eqref{eq:7}. For the binary tensor,
  given $F,G \in [\CR, \CX]$, we form the pointwise tensor
  $H = \Gamma F \otimes \Gamma G$ with
  \begin{equation*}
    HA = (\Gamma F) A \otimes (\Gamma G)A = \bigoplus_{B \leqslant_n A} FB \otimes \bigoplus_{B' \leqslant_{n'} A} GB' \cong
    \bigoplus_{\substack{B \leqslant_n A\\ B' \leqslant_{n'} A}} FB \otimes GB'
  \end{equation*}
  on objects; we will not need the full definition on morphisms, but we
  see by the observations above that, for any $\CM$-map
  $m \colon R \rightarrowtail A$, the idempotent
  $H(m,m) \colon HA \rightarrow HA$ is given by projection onto those
  summands $(n,n') \in \mathrm{Sub}(A)^2$ for which both $n$ and $n'$
  factor through $m$. Comparing with~\eqref{eq:7}, we conclude that
  $(N H)(A)$ is given by restricting $HA$ to those direct summands
  $(n,n')$ which do \emph{not} have this property for any proper
  $m \in \mathrm{Sub}(A)$: which gives the formula for $F \ast G$
  displayed above.
\end{proof}

\begin{examples}
  \label{ex:1}
  \begin{enumerate}[(i)]
  \item Let $\CC$ be the category of finite sets and injections, and
    $(\CE, \CM)$ the (isomorphisms, all maps) factorisation system.
    Then $\CR$ is the groupoid $\mathfrak S$ of finite sets and
    bijections and $\CP$ is the category of finite sets and partial
    injections, denoted by $\mathrm{FI}\sharp$
    in~\cite{Church2015FI-modules}; the equivalence
    $[\mathfrak{S}, \CX] \rightarrow [\mathrm{FI}\sharp, \CX]$ for any
    additive Karoubian $\CX$ is now that of Theorem~4.1.5 of
    \emph{ibid}. When $\CX$ is an $\mathbf{Ab}$-algebra, the induced
    tensor product on $[\mathfrak S, \CX]$ is the case
    $\Lambda = \mathbb Z$ of~\eqref{Lten}, so that this tensor product
    corresponds to the pointwise tensor product of
    $\mathrm{FI}\sharp$-modules.\vskip0.25\baselineskip
  \item In a similar way, when $\CC$ is $\Delta^+_{\mathrm{inj}}$, the
    category of finite ordinals $[n]$ and monotone injections, and
    $(\CE, \CM) = (\text{identities}, \text{all maps})$, we have
    $\CR = \mathbb N$ and $\CP$ the category $\mathrm{FO}\sharp$ of
    finite ordinals and partial monotone injections; in this way, we
    re-find for any $\mathbf{Ab}$-algebra $\CX$ the case
    $\Lambda =\mathbb Z$ of the tensor product~\eqref{Nten} on
    $\CX^\mathbb N$, but now with the additional information that it
    is monoidally equivalent to the pointwise monoidal structure on
    $[\mathrm{FO}\sharp, \CX]$.\vskip0.25\baselineskip
  \item Take $\CC$ itself to be $\mathrm{FO}\sharp$, and let $\CE$ and
    $\CM$ comprise the maps therein with entire codomain and domain
    respectively. Then $\CR$ is isomorphic to
    $(\Delta^+_{\mathrm{inj}})^\mathrm{op}$, while $\CP$ is the cube
    category $\mathbb I$ of~\cite{Crans1995Pasting}, diagrams on which
    are the cubical sets\footnote{With degeneracies but without
      symmetries or connections.} of~\cite{Kan1955}. For any
    $\mathbf{Ab}$-algebra $\CX$, the pointwise product on cubical
    objects in $\CX$ thus transports to a Hurwitz-style tensor product
    on the category $[(\Delta^+_\mathrm{inj})^\mathrm{op}, \CX]$ of
    augmented semi-simplicial objects in $\CX$.\vskip0.25\baselineskip
  \item Let $\CC$ be the category of finite sets and \emph{all} maps,
    and $(\CE, \CM)$ the (epi, mono) factorisation system. In this
    case, $\CP$, the category of finite sets and partial maps, is
    isomorphic to Segal's category $\Gamma$ of finite pointed sets,
    while $\CR$ is the category $\Omega$ of finite sets and
    epimorphisms. The equivalence
    $[\Omega, \CX] \rightarrow [\Gamma, \CX]$ for any additive
    Karoubian $\CX$ was described
    in~\cite[Theorem~3.1]{Pirashvili2000Dold-Kan}, and our
    construction yields a Hurwitz-style tensor product on
    $[\Omega, \CX]$ for any $\mathbf{Ab}$-algebra $\CX$.
  \end{enumerate}
\end{examples}

We now summarise the analogue of the preceding results for the general
case of~\eqref{eq:2} as described in~\cite{Lack2015Combinatorial}. The
basic data are a category $\CP$; a subcategory $\CM$ containing all
the isomorphisms; and an identity-on-objects functor
$(\thg)^\ast \colon \CM^\mathrm{op} \rightarrow \CP$ such that
$m^\ast \circ m = 1$ for every $m \in \CM$ (so in particular, each $m$ is
a split monomorphism in $\CC$). The class of morphisms $\CR$ is
defined to comprise those $r \in \CP$ such that, if
$r = m \circ x \circ n^\ast$ for any $m,n \in \CM$, then $m$ and $n$
are invertible. The category $\CD$ has the same objects as $\CP$, and
as morphisms the maps in $\CR$ with a zero morphism freely adjoined
between any pair of objects; if $r, s \in \CR$ are composable
morphisms in $\CD$, then their composite is $s \circ r$ if this also
lies in $\CR$, and is zero otherwise.

These data are required to satisfy various Assumptions which are
listed in~\cite[Section~2]{Lack2015Combinatorial}; one of these is
that every morphism $f \in \CP$ factors as
$f = n \circ r \circ m^\ast$, uniquely up to isomorphism, for
$m, n \in \CM$ and $r \in \CR$, while another is that each set
$\mathrm{Sub}(A)$ of $\CM$-subobjects is finite. Given these data,
Theorem~6.7 of \emph{ibid}.\ defines the equivalence $\Gamma$
of~\eqref{eq:2} as follows. For $F \in [\CD, \CX]_\mathrm{pt}$, we
take $\Gamma F \colon \CP \rightarrow \CX$ to be given on objects as
before by
\begin{equation*}
  (\Gamma F)A = \bigoplus_{B \leqslant_n A} FB\rlap{ ,}
\end{equation*}
where the sum is again over the (finite) set of $\CM$-subobjects of
$A$. For its value at a map $f \colon A \rightarrow A'$ in $\CP$,
suppose that $B \leqslant_n A$ and that $B' \leqslant_{n'} A'$, and
define $\xi(f)_{nn'} \colon FB \rightarrow FB'$ to be $Fr$ if there is
a (necessarily unique) diagram of the form
\begin{equation*}
  \vcenter{\hbox{\xymatrix{
    B \ar@{ >->}[d]_-{n} \ar@{->}[r]^-{r \in \CR} & B' \ar@{ >->}[d]^-{n'}\\
    A \ar[r]^-{f} & A' 
  }}}
\end{equation*}
and to be zero otherwise; we now take $(\Gamma F)(f)$ to be the matrix
with entries $\xi(f)_{nn'}$. It follows from the assumptions that, for
any $m \in \CM$, the map $(\Gamma F)(mm^\ast)$ is the idempotent on
$\bigoplus_{B \leqslant_n A} FB$ projecting onto those summands
$n \in \mathrm{Sub}(A)$ for which $m^\ast n \in \CM$. The
pseudoinverse $N$ to $\Gamma$ is once again defined by the
formula~\eqref{eq:7} (with $H(mm^\ast)$ replacing $H(m,m)$) and now
tracing through the remainder of the argument given above yields the
following more general version of Proposition~\ref{prop:8}. When
interpreting~\eqref{eq:8} in this context, we say that
$n, n' \in \mathrm{Sub}(A)$ are \emph{covering} if whenever
$m \in \CM$ is such that $m^\ast n$ and $m^\ast n'$ are both in $\CM$,
then $m$ is invertible.

\begin{proposition}\label{prop:8}
  Suppose given a category $\CP$, subcategory $\CM$, and
  identity-on-objects
  $(\thg)^\ast \colon \CM^\mathrm{op} \rightarrow \CP$ satisfying the
  Assumptions of~\cite{Lack2015Combinatorial}, and let $\CD$ be the
  associated category with zero maps as above. For any (symmetric)
  $\mathbf{Ab}$-algebra $\CX$ there is a (symmetric)
  $\mathbf{Ab}$-algebra structure $(\ast, J)$ on
  $[\CD, \CX]_\mathrm{pt}$ defined as in~\eqref{eq:8}. Moreover,
  \eqref{eq:2} is strong monoidal as a functor
  $\Gamma \colon ([\CD, \CX]_\mathrm{pt}, \ast, J) \rightarrow ([\CP,
  \CX], \text{pointwise})$.
\end{proposition}

The main additional example that this more general case allows is the
following one. If we take $\CP$ to be $\Delta^\mathrm{op}$, for
$\Delta$ the category of non-empty finite ordinals and monotone maps,
take $\CM$ to comprise the surjective monotone maps, and for each
$m \in \CM$ take $m^\ast$ to be its right adjoint, then the
assumptions of~\cite{Lack2015Combinatorial} may be verified to hold as
in Example~3.3 of \emph{ibid}. In this case, the category $\CD$ turns
out to be the indexing category for chain complexes, and the
equivalence~\eqref{eq:2} the classical Dold--Kan equivalence between
simplicial objects and chain complexes in any additive Karoubian
$\CX$. The preceding proposition now describes for any
$\mathbf{Ab}$-algebra $\CX$ a tensor product of chain complexes in
$\CX$ given by
\begin{equation*}
  (A \ast B)(n) = \bigoplus_{\substack{\sigma \colon [k] \twoheadleftarrow [n] \twoheadrightarrow [\ell] \colon \tau \\ \sigma, \tau \text{ jointly monic}}} Ak \otimes B\ell\rlap{ .}
\end{equation*}
This tensor product was calculated explicitly in unpublished work \cite{LackTalk}
of Lack and Hess; the key point is that $A \ast B$ is a well-behaved
retract of the usual tensor product $A \otimes B$ of chain complexes.
This can be understood as part of the fact that the equivalence
$N \colon [\Delta^\mathrm{op}, \CX] \rightarrow \mathbf{Ch}(\CX)$ is a
Frobenius monoidal functor, as explained in
\cite[Chapter~5]{AguiarMahajan2010}.

\section{Dold--Kan equivalences from small coalgebras}
\label{sec:dold-kan-equiv}

We conclude this paper by discussing coalgebraic aspects of the
equivalences described in the previous section. To this end, we
consider the $2$-category $\mathbf{Ab}\text-\mathbf{Cat}_{cc}$ whose
objects are additive Karoubian categories, and whose $1$- and
$2$-cells are $\mathbf{Ab}$-enriched functors and transformations.
$\mathbf{Ab}\text-\mathbf{Cat}_{cc}$ is a symmetric monoidal biclosed
bicategory: the tensor product is obtained by Cauchy completing the
usual $\mathbf{Ab}$-categorical tensor product, while the internal hom
is the standard $\mathbf{Ab}$-enriched functor category. We can thus
talk about (symmetric) monoidales and comonoidales in
$\mathbf{Ab}\text-\mathbf{Cat}_{cc}$; the monoidales are the
$\mathbf{Ab}$-algebras considered previously, while the comonoidales
we refer to as \emph{$\mathbf{Ab}$-coalgebras}. Much as in
Section~\ref{sec:lambda-weight-tens}, we have a free--forgetful
biadjunction
$\mathbf{Ab}\text-\mathbf{Cat}_{cc} \leftrightarrows \mathbf{Cat}$
whose left biadjoint is a strong monoidal homomorphism, and we reuse
the notation $\spn{\CA}$ for the free additive Karoubian category on
$\CA \in \mathbf{Cat}$. If $\CA$ is a category with zero morphisms,
then we write $\spn{\CA}_\mathrm{pt}$ for the free additive Karoubian
category on $\CA$ \emph{qua} category with zero morphisms.

Suppose now that we are given a category $\CP$, a subcategory $\CM$
and a functor $(\thg)^\ast \colon \CM^\mathrm{op} \rightarrow \CP$
satisfying the assumptions of~\cite{Lack2015Combinatorial}, and let
$\CD$ be as before the associated category with zero maps. We may see
the equivalences $\Gamma \colon [\CD, \CX]_\mathrm{pt} \to [\CP, \CX]$
of~\eqref{eq:2} for any additive Karoubian $\CX$ as induced by
precomposition with an equivalence
$\spn{\CP} \rightarrow \spn{\CD}_\mathrm{pt}$ of additive Karoubian
categories, defined on generating objects by
$A \mapsto \bigoplus_{B \leqslant_n A} B$. The point we wish to make
is that, for several of our examples above, the
$\mathbf{Ab}$-categories $\spn{\CP}$ and $\spn{\CD}_\mathrm{pt}$ are
$\mathbf{Ab}$-algebras which are free on pointed additive Karoubian
categories, and the equivalence between them generated by an equivalence 
at this more primitive level. For our first example of this, we consider the
equivalence $[\mathfrak S, \CX] \simeq [\mathrm{FI}\sharp, \CX]$ of
Examples~\ref{ex:1}(i). We write $\mathbf{SE}$ for the \emph{free
  split epimorphism} as displayed in:
\begin{equation*}
  \xymatrix{
    E \ar[r]^{m} \ar[dr]_{1} & D \ar[d]^{e} \\
    & E \rlap{ .}
  }
\end{equation*}

\begin{proposition}
  \label{prop:9}
  The free symmetric $\mathbf{Ab}$-algebras on the pointed additive
  Karoubian categories $\spn{E} \rightarrow \spn{E,D}$ and
  $\spn{E} \rightarrow \spn{\mathbf{SE}}$ are respectively
  $\spn{\mathfrak S}$ and $\spn{\mathrm{FI}\sharp}$ under the monoidal
  structures given on basis elements by disjoint union.
\end{proposition}

\begin{proof}
  The first claim follows as in Proposition~\ref{prop:5} above, and
  the second does too once we observe that $\mathrm{FI}\sharp$ is the
  free symmetric monoidal category on the pointed category
  $\{E\} \rightarrow \mathbf{SE}$, with the generators $E$ and $D$
  corresponding to the empty and singleton sets $0$ and $1$.
\end{proof}

It is straightforward to see that $\spn{E, D}$ and $\spn{\mathbf{SE}}$
are equivalent as pointed $\mathbf{Ab}$-categories. In one direction,
we have the $\mathbf{Ab}$-functor
$\Theta \colon \spn{\mathbf{SE}} \rightarrow \spn{E,D}$ classifying the
split epimorphism $\pi_1 \colon E \oplus D \rightarrow E$ with section
$\iota_1 \colon E \rightarrow E \oplus D$. In the other, we have the
$\mathbf{Ab}$-functor $\spn{E,D} \rightarrow \spn{\mathbf{SE}}$
picking out the pair of objects $E$ and
$\mathrm{ker}(e \colon D \rightarrow E)$ (here, as before, this kernel
can be constructed as a splitting of the idempotent $1-me$ in
$\spn{\mathbf{SE}}$). Applying Proposition~\ref{prop:9}, we deduce the
existence of a strong monoidal equivalence $\Theta_\infty$ fitting into a
pseudocommutative square
\begin{equation*}
  \xymatrix{
    {\spn{\mathbf{SE}}} \ar[r]^-{\Theta} \ar[d]_{} &
    {\spn{E,D}} \ar[d]^{} \\
    {\spn{\mathrm{FI}\sharp}} \ar[r]^-{\Theta_\infty} &
    {\spn{\mathfrak S}\rlap{ ,}}
  }
\end{equation*}
and so by composing with $\Theta_\infty$ an equivalence
$[\mathfrak S, \CX] \simeq [\mathrm{FI}\sharp, \CX]$, which a direct
analysis shows is precisely the equivalence of~\eqref{eq:6}.

Moreover, the induced (symmetric) $\mathbf{Ab}$-algebra structures on
$[\mathfrak S, \CX]_\mathrm{pt}$ and $[\mathrm{FI}\sharp, \CX]$ when
$\CX$ is a (symmetric) $\mathbf{Ab}$-algebra may be seen as induced by
convolution with symmetric $\mathbf{Ab}$-coalgebra structures on
$\spn{\mathfrak S}$ and $\spn{\mathrm{FI}\sharp}$. As before, these
structures make $\spn{\mathfrak S}$ and $\spn{\mathrm{FI}\sharp}$ into
symmetric \emph{$\mathbf{Ab}$-bialgebras}; and as before, these
symmetric bialgebra structures are in fact freely generated from
pointed $\mathbf{Ab}$-coalgebra structures on $\spn{E,D}$ and
$\spn{\mathbf{SE}}$ respectively. Explicitly, we equip $\spn{E,D}$
with the diagonal coalgebra structure---so
$\varepsilon(E) = \varepsilon(D) = \mathbb Z$ and
$\delta(E) = E \otimes E$, $\delta(D) = D \otimes D$---and endow
$\spn{\mathbf{SE}}$ with the coalgebra structure given on generating
objects by
\begin{equation*}
  \varepsilon(E)= \mathbb Z \ , \ \varepsilon(D)=0\ , \ \delta(E)=E\otimes E \ , \ \delta(D)=D\otimes E \oplus E\otimes D \oplus D\otimes D
\end{equation*}
and on generating morphisms in the unique possible manner. It is quite
straightforward to see that, with respect to these structures, $\Theta$
becomes a strong morphism of $\mathbf{Ab}$-coalgebras, and so that, by
the argument of Section~\ref{sec:lambda-weight-tens} above,
$\Theta_\infty$ is an equivalence of symmetric $\mathbf{Ab}$-bialgebras.
As in Theorem~\ref{thm:1} above, the induced coalgebra structure on
$\spn{\mathfrak S}$ is the diagonal one, so that the induced monoidal
structure on $[\mathfrak S, \CX]$ is pointwise. By uniqueness of
transport of monoidal structure, it follows that the coalgebra
structure on $\spn{\mathrm{FI}\sharp}$ must be the one which, under
convolution, induces the Hurwitz product on each
$[\mathrm{FI}\sharp, \CX]$. In the situation just discussed, we may
also take free \emph{non-symmetric} monoidales; whereupon the
equivalence $\Theta \colon \spn{E, D} \rightarrow \spn{\mathbf{SE}}$ of
pointed $\mathbf{Ab}$-coalgebras yields an equivalence of
$\mathbf{Ab}$-bialgebras
$\spn{\mathbb N} \rightarrow \spn{\mathrm{FO}\sharp}$, inducing the
(monoidal) equivalences of Examples~\ref{ex:1}(ii).

Finally, let us show how Examples~\ref{ex:1}(iii) may be obtained in a
corresponding manner. If we consider the arrow category
$\mathbf 2 = \{f \colon D \rightarrow E\}$ and the free category
$\mathbf{G}$ on a reflexive graph, as to the left in:
\begin{equation*}
  \xymatrix{
    E \ar[r]^{i} \ar[dr]_{1} & D \ar@<3pt>[d]^-{t} \ar@<-3pt>[d]_-{s} \\
    & E 
  } \qquad \qquad 
  \xymatrix@+0.3em{
    E \ar[r]^-{\iota_1} \ar[dr]_{1} & E \oplus D \ar@<3pt>[d]^(0.4){(-1 \ f)} \ar@<-3pt>[d]_(0.4){(1\ 0)} \\
    & E \rlap{ .}
  }
\end{equation*}
As is well-known, there is an equivalence of additive Karoubian
categories $\Theta \colon \spn{\mathbf{G}} \rightarrow \spn{\mathbf 2}$
whose action on generators picks out the reflexive graph as right
above. With each category pointed by the object $E$, it is easy to see
that the free monoidal categories thereon are given by
$(\Delta^+_\mathrm{inj})^\mathrm{op}$ and the cube category
$\mathbb I$ respectively, and so we obtain an equivalence of
$\mathbf{Ab}$-algebras
$\spn{(\Delta^+_\mathrm{inj})^\mathrm{op}} \simeq \spn{\mathbb I}$
inducing the equivalences of Examples~\ref{ex:1}(iii) above. Once
again, the equivalence between $\spn{\mathbf G}$ and $\spn{\mathbf 2}$
may be made into one of pointed $\mathbf{Ab}$-coalgebras in such a way
that, on passing to the associated free $\mathbf{Ab}$-bialgebras, we
reconstruct the monoidal equivalence
$[\mathbb I, \CX] \simeq [(\Delta_\mathrm{inj}^+)^\mathrm{op}, \CX]$
for any $\mathbf{Ab}$-algebra $\CX$.

\begin{center}
--------------------------------------------------------
\end{center}

\appendix


\begin{thebibliography}{000} 

\bibitem{AguiarMahajan2010} Marcelo Aguiar and Swapneel Mahajan, \textit{Monoidal functors, species and Hopf algebras}, CRM Monograph Series \textbf{29} (Amer. Math. Soc., 2010). \label{AguiarMahajan2010} 

\bibitem{Barr1970} Michael Barr, \textit{Coequalizers and free triples}, Math. Zeitschrift \textbf{116} (1970) 307--322.\label{Barr1970}

\bibitem{Baxt1960} Glen Earl Baxter, \textit{An analytic problem whose solution follows from a simple algebraic identity}, Pacific Journal Math. \textbf{10} (1960) 731--742.\label{Baxt1960}

\bibitem{Blackwell1976} Robert Blackwell, \textit{Some existence theorems in the theory of doctrines}, (Univ. New South Wales, Ph.D. Thesis,1976).\label{Blackwell1976}

\bibitem{111} Thomas Booker and Ross Street, \textit{Tannaka duality and convolution for duoidal categories}, Theory and Applications of Categories \textbf{28} (2013)166--205.\label{111}

\bibitem{Borceux1994} Francis Borceux, \textit{Handbook of categorical
    algebra 2}, Encyclopedia of Mathematics and its Applications
  \textbf{51} (Cambridge University Press, 1994).\label{Borceux1994}

\bibitem{Cart1972} Pierre Cartier, \textit{On the Structure of Free Baxter Algebras}, Advances in Math. \textbf{9} (1972) 253--265.\label{Cart1972}

\bibitem{Church2015FI-modules}
T.~Church, J.~S. Ellenberg, and B.~Farb, \emph{F{I}-modules and
  stability for representations of symmetric groups}, Duke
Mathematical Journal, \textbf{164(9)} (2015) 1833--1910.

\bibitem{Crans1995Pasting} Sjoerd Crans, {\em Pasting schemes for the
    monoidal biclosed structure on {$\omega$}-{C}at}, \newblock PhD
  thesis, Utrecht University, 1995.

\bibitem{DayPhD} Brian J. Day, Construction of Biclosed Categories (PhD Thesis, UNSW, 1970) \url{http://www.math.mq.edu.au/~street/DayPhD.pdf}.\label{DayPhD} 

\bibitem{DayConv} Brian J. Day, \textit{On closed categories of functors}, Lecture Notes in Mathematics \textbf{137} (Springer-Verlag, 1970) 1--38.\label{DayConv}

\bibitem{55} Brian J. Day and Ross Street, \textit{Kan extensions along promonoidal functors}, Theory and Applications of Categories \textbf{1(4)} (1995) 72--78.\label{55}

\bibitem{mbaHa} Brian J. Day and Ross Street, \textit{Monoidal bicategories and Hopf algebroids}, Advances in Math. \textbf{129} (1997) 99--157.\label{mbaHa}

\bibitem{Dold1961Homologie} Albrecht Dold and Dieter Puppe, \newblock
  \emph{Homologie nicht-additiver {F}unktoren. {A}nwendungen},
  \newblock Universit{\'e} de Grenoble. Annales de l'Institut Fourier.
  \textbf{11} (1961) 201--312.

\bibitem{Dubuc1974} Eduardo Dubuc, \textit{Free monoids}, Journal of Algebra, \textbf{29} (1974) 208--228.\label{Dubuc1974}

\bibitem{EilKel1966} Samuel Eilenberg and G. Max Kelly, \textit{Closed categories}, Proceedings of the Conference on Categorical Algebra (La Jolla, 1965), (Springer-Verlag,1966) 421--562.\label{EilKel1966}

\bibitem{FreydKelly} Peter J. Freyd and G. Max Kelly, \textit{Categories of continuous functors I}, J. Pure and Applied Algebra \textbf{2} (1972) 169--191; Erratum Ibid. \textbf{4} (1974) 121.\label{FreydKelly}

\bibitem{GrandisPare1999} Marco Grandis and Robert Par\'e, Limits in double categories, Cahiers de Topologie et G\'eom\'etrie Diff\'erentielle Cat\'egoriques \textbf{40} (1999) 162--220.\label{GrandisPare1999}

\bibitem{GK2000} Li Guo, William Keigher and Shilong Zhang, \textit{Baxter algebras and shuffle products}, Advances in Math. \textbf{150} (2000) 117--149.\label{GK2000}

\bibitem{GKZ2008} Li Guo and William Keigher, \textit{On differential Rota-Baxter algebras}, Journal of Pure and Applied Algebra \textbf{212} (2008) 522--540.\label{GKZ2008}

\bibitem{GKZ2014} Li Guo, William Keigher and Shilong Zhang, \textit{Monads and distributive laws for Rota-Baxter and differential algebras}, \url{arXiv:1412.8058v2}.\label{GKZ2014}

\bibitem{JacobsonBAII} Nathan Jacobson, ``Basic Algebra II'' (Dover Books on Mathematics, Second Edition, reprinted 2012).\label{JacobsonBAII}

\bibitem{Species} Andr\'e Joyal, \textit{Une th\'eorie combinatoire des s\'eries formelles}, Advances in Mathematics \textbf{42} (1981) 1--82.\label{Species}

\bibitem{AnalFunct} Andr\'e Joyal, \textit{Foncteurs analytiques et esp\`eces de structures}, Lecture Notes in Mathematics \textbf{1234} (Springer 1986) 126--159.\label{AnalFunct} 

\bibitem{TYBO} Andr\'e Joyal and Ross Street, \textit{Tortile Yang-Baxter operators in tensor categories}, Journal of Pure and Applied Algebra \textbf{71} (1991) 43--51.\label{TYBO} 

\bibitem{BTC}  Andr\'e Joyal and Ross Street, \textit{Braided tensor categories}, Advances in Mathematics \textbf{102} (1993) 20--78.\label{BTC}

\bibitem{GLFq} Andr\'e Joyal and Ross Street, \textit{The category of representations of the general linear groups over a finite field}, Journal of Algebra \textbf{176} (1995) 908--946.

\bibitem{Kan1955} Daniel Kan, \emph{Abstract homotopy. I}, Proceedings of the National
  Academy of Sciences of the United States of America, \textbf{41} (1955) 1092--1096.

\bibitem{Kapr1995} Mikhail M.\,Kapranov, \textit{Analogies between the Langlands correspondence and topological quantum field theory}, Progress in Mathematics \textbf{131} (Birkh\"auser Boston, Massachusetts, 1995) 119--151.

\bibitem{Kellytransfinite} G. Max Kelly, \textit{A unified treatment of transfinite construction for free algebras, free monoids, colimits, associated sheaves, and so on}, Bulletin Austral. Math. Soc. \textbf{22} (1980) 1--83;
\textbf{26} (1982) 221--237.\label{Kellytransfinite}

\bibitem{KellyBook} G. Max Kelly, \textit{Basic concepts of enriched category theory}, London Mathematical Society Lecture Note Series \textbf{64} (Cambridge University Press, Cambridge, 1982). \label{KellyBook}

\bibitem{KelSt1974} G. Max Kelly and Ross Street, \textit{Review of the elements of 2-categories}, Lecture Notes in Mathematics \textbf{420} (Springer-Verlag, 1974) 75--103.\label{KelSt1974}

\bibitem{LackTalk}
Stephen Lack, 
\emph{The {D}old-{K}an correspondence and tensor products of chain complexes},
Talk at the Australian Category Seminar (Macquarie University, 27 October 2010).

\bibitem{Lack2015Combinatorial}
Stephen Lack and Ross Street, 
\emph{Combinatorial categorical equivalences of {D}old-{K}an type},
Journal of Pure and Applied Algebra, \textbf{219(10)} (2015) 4343--4367.

\bibitem{CWM} Saunders Mac Lane, \textit{Categories for the Working Mathematician}, Graduate Texts in Mathematics \textbf{5} (Springer-Verlag, 1971).\label{CWM}

\bibitem{Pirashvili2000Dold-Kan} Teimuraz Pirashvili, \emph{Dold-{K}an type theorem for {$\Gamma$}-groups},
Mathematische Annalen, \textbf{318(2)} (2000) 277--298.

\bibitem{Rota1969} Gian-Carlo Rota, \textit{Baxter algebras and combinatorial identities I}, Bulletin Amer. Math. Soc. \textbf{5} (1969) 325--329.\label{Rota1969}

\bibitem{Slominska} Jolanta S{\l }omi{\' n}ska, \emph{Dold-{K}an type
    theorems and {M}orita equivalences of functor categories}, Journal
  of Algebra \textbf{274} (2004) 118--137.

\bibitem{UCL} Ross Street, \textit{Monoidal categories in, and linking, geometry and algebra}, Bulletin of the Belgian Mathematical Society -- Simon Stevin \textbf{19(5)} (2012) 769--821.

\bibitem{scc} Ross Street, \textit{Skew-closed categories}, Journal of Pure and Applied Algebra \textbf{217(6)} 
(2013) 973--988.\label{scc}  

\bibitem{wtp} Ross Street, \textit{Weighted tensor products of Joyal species, graphs, and charades}, \url{http://arxiv.org/abs/1503.02783}.\label{wtp}

\end{thebibliography}
\end{document}